\pdfoutput=1
\documentclass{article}
\usepackage{amsfonts,amsbsy,amsmath,amssymb, amsthm}
\usepackage{hyperref}  

\newtheorem{theorem}{Theorem}[section]
\newtheorem{lemma}[theorem]{Lemma}
\newtheorem{observation}[theorem]{Observation}

\newtheorem{proposition}[theorem]{Proposition}
\newtheorem{definition}[theorem]{Definition}
\newtheorem{claim}[theorem]{Claim}
\def\QED{$\blacksquare$}
\def\inQED{$\square$}

\renewenvironment{proof}
{\vspace{1ex}\noindent{\sl Proof.}\hspace{0.5em}}{\hfill \QED \vspace{1ex}}

\newenvironment{proofof}[1]
{\vspace{1ex}\noindent{\sl Proof of #1.}\hspace{0.5em}}{\hfill \QED \vspace{1ex}}

\newenvironment{innerproofof}[1]
{\vspace{1ex}\noindent{\sl Proof of #1.}\hspace{0.5em}}{\hfill \inQED \vspace{1ex}}

\newcounter{thesubclaim}
\newenvironment{subclaim}
{%
\refstepcounter{thesubclaim} 
\medbreak\noindent
\begin{it}%
\noindent
{\bf Subclaim \arabic{thesubclaim}.}\hspace{-0.5ex}
}
{\end{it} \smallskip}

\let\mcal=\mathcal
\let\mathds=\mathbb

\newcommand{\B}{\mcal{B}}
\newcommand{\C}{\mcal{C}}

\newcommand{\F}{\mcal{F}}
\newcommand{\G}{\mcal{G}}

\newcommand{\I}{\mcal{I}}

\newcommand{\sS}{\mcal{S}}

\newcommand{\U}{\mcal{U}}
\newcommand{\V}{\mcal{V}}

\newcommand{\dR}{\mathds{R}}

\let\theta=\vartheta
\let\rho=\varrho
\let\sigma=\varsigma
\let\phi=\varphi

\newcommand{\sm}{\setminus}

\def\td{tree-decom\-po\-si\-tion}

\newcommand{\COMMENT}[1]{}
\DeclareMathOperator{\del}{del}

\begin{document}
\title{The structure of 2-separations of infinite matroids}
\author{Elad Aigner-Horev \thanks{Mathematics and Computer Science Department, Ariel University, Israel} \and Reinhard Diestel \thanks{Mathematisches Seminar, Hamburg University, Germany} \and Luke Postle \thanks{Combinatorics and Optimisation Department, University of Waterloo, Canada}}
\date{}
\maketitle

\begin{abstract}\noindent
Generalizing a well known theorem for finite matroids, we prove that for every (infinite) connected matroid $M$ there is a unique tree $T$ such that the nodes of $T$ correspond to minors of $M$ that are either 3-connected or circuits or cocircuits, and the edges of $T$ correspond to certain nested 2-separations of $M$. These decompositions are invariant under duality. 
\end{abstract}

\section{Introduction}\label{intro}
A well known theorem of Cunningham and Edmonds~\cite{CE}, proved independently also
by Seymour \cite{Seymour}, states that for every connected finite matroid $M$ there is a
unique tree $T$ such that the nodes of $T$ correspond to minors of $M$ each of
which is either $3$-connected, a circuit, or a cocircuit,
and the edges of $T$ correspond to certain 2-separations of $M$.

Cunningham and Edmonds also prove that, given such a decomposition tree for $M$ with an assignment
of minors and separations of $M$ to its nodes and edges,
the same tree with the minors replaced by their duals defines a decomposition
tree for the dual of $M$ \cite{CE,OxleyBook}.

Richter~\cite{R04} proved that infinite $2$-connected graphs admit such a decomposition.

Our aim in this paper is to extend these results to infinite matroids, not necessarily finitary. This is less straightforward than the finite case, for two reasons. One is that we have to handle connectivity formally differently, without using rank. A~more fundamental difference is that we cannot obtain the desired parts of our decomposition simply by decomposing the matroid recursively, since such a recursion might be transfinite and end with limits beyond our control. Instead, we shall define the parts explicitly, and will then have to show that they do indeed make up the entire matroid and fit together in the desired tree-structure. This is outlined in more detail in Section~\ref{sec:outline}.

Our result has become possible only by the recent axiomatization of infinite matroids with duality~\cite{InfMatroidAxioms}. This has already prompted a number of generalizations of standard finite matroid theorems to infinite matroids~\cite{AfzaliBowler12, A-HCF11:matroidunion, A-HCF11:matroidintersection, ALM15:InfiniteGammoids, B13:TutteConnectivity, BC12:excluded_matroid_minors, BC12:wildmatroids, BowlerCarmesinMatroidIntersection, BC13:Determinacy, BC13:Ubiquity, BC14:Trees, BC14:IntersectionTrees, BCC:graphic_matroids, BruhnDiestelMatroidsGraphs, BruhnWollanConInfMatroids, C13:bureaucracy, C14:gammoids, C14:CycleMatroids}. The result we prove here appears to be the first such generalization to all matroids, without any assumptions of finitariness, co-finitariness, or a combination of these.

\subsection{Connectivity of infinite matroids} A matroid $M$ is \emph{connected} if every two of its elements lie in a common circuit. Higher-order connectivity for finite matroids is usually defined via the rank function, which is not possible for infinite matroids. However, there is a natural rank-free reformulation, as follows.

Consider a partition $(X,Y)$ of the ground set of a matroid~$M$, with the sets $X$ and $Y$ possibly empty. Given a basis $B_X$ of $M|X$ and a basis $B_Y$ of $M|Y$, the matroid $M$ will be spanned by $B_X\cup B_Y$, so there exists a set $F\subseteq B_X\cup B_Y$ such that $(B_X\cup B_Y) \sm F$ is a basis of~$M$. Bruhn and Wollan~\cite{BruhnWollanConInfMatroids} showed that the size $k = |F|$ of this set does not depend on the choices of $B_X$, $B_Y$ and~$F$, but only on~$(X,Y)$. If in addition $|X|,|Y| \geq k+1$, we call $(X,Y)$ a {\em separation\/} of~$M$, or more specifically a {\em $(k+1)$-separation\/}\footnote{Some authors, including Oxley~\cite{OxleyBook}, call this an {\em exact\/} $(k+1)$-separation, and use the term `$(k+1)$-separation' for any separation of order at most~$k+1$. The tradition of referring to~$k+1$, rather than~$k$, as the {\em order\/} of a separation with $|F|=k$ may be regrettable but is standard.}
or a {\em separation of order\/}~$k+1$. The matroid $M$ is {\em $n$-connected\/} if it has no $\ell$-separation for any $\ell<n$. For $M$ finite, these definitions are equivalent to the traditional ones.

\noindent
\subsection{Tree-decompositions} Let $T$ be a tree. Consider a partition $R=(R_v)_{v\in T}$ of the ground set $E$ of a matroid $M$ into {\em parts\/} $R_v$, one for every node $v$ of~$T$. (We allow $R_v = \emptyset$.) Given an edge $e=vw$ of~$T$, write $T_v$ and $T_w$ for the components of $T-e$ containing $v$ and~$w$, respectively, and put $S(e,v) := \bigcup_{u\in T_v} R_u$ and $S(e,w) := \bigcup_{u\in T_w} R_u$. If each of the partitions $(S(e,v), S(e,w))$ of $E$, as $vw$ varies over the edges of~$T$, is a separation of~$M$, we call the pair $(T,R)$ a {\em \td\/} of~$M$. The supremum of the orders of the separations $(S(e,v), S(e,w))$ is the \emph{adhesion} of the decomposition $(T,R)$. If all these separations have the same order~$k$, then we say that $(T,R)$ has \emph{uniform adhesion}~$k$.%

Let $(T,R)$ be a tree-decomposition of $M$ of uniform adhesion~2. With every node $v\in T$ we shall associate a matroid~$M_v$, whose ground set will be the set $R_v$ together with some `virtual elements', one for every edge of $T$ at~$v$. Formally, let the ground set of $M_v$ be the set $R_v\cup F_v$, where $F_v$ is the set of all the edges of $T$ incident with~$v$. As the {\em circuits\/} of $M_v$ we take the sets
 \begin{equation}\label{eq:local-cir} 
  (C\cap R_v)\cup \{e\in F_v\mid e=vw\text{ with } C\cap S(e,w)\ne\emptyset\},
 \end{equation}
where $C$ ranges over all the circuits of $M$ not contained in any of the sets $S(vw,w)$.
We shall prove in Lemma~\ref{localmatroid1} that 
\begin{equation} \label{localmatroid}
\mbox{$M_v$ is a matroid on $R_v\cup F_v$.}
\end{equation}

\noindent
We call the matroids $M_v$ the \emph{torsos} of the tree-decomposition $(T,R)$.

As we shall see below, if a torso $M_v$ is a circuit of size at least~4 we can partition $R_v$ into two subsets, and correspondingly split $v$ into adjacent nodes $v_1,v_2$ of~$T$, to obtain another tree-decomposition of~$M$ of uniform adhesion~2; in this \td, $M_{v_1}$ and $M_{v_2}$ will again be circuits. This split of $R_v$ can be done in more than one way. Hence even if we aim to make the sets $R_v$ as small as possible, our tree-decomposition of $M$ of uniform adhesion~2 will not in general be unique.

\goodbreak

To achieve uniqueness, we therefore forbid `adjacent' cycles and cocyles, as follows. Call a tree-decomposition $(T,R)$ \emph{irredundant} if
\begin{enumerate}
\item[(i)] all torsos have size at least three; and
\item[(ii)] for every edge $vw$ of~$T$, the torsos $M_v, M_w$ are not both circuits and not both cocircuits.
\end{enumerate}

\subsection{Statement of results}

The following infinite decomposition theorem is our main result:

\begin{theorem}\label{main}
\emph{}\vskip-6pt\vskip-0pt
\begin{enumerate}
\item[\rm(i)] Every connected matroid with at least three elements, finite or infinite, has an irredundant tree-decomposition of uniform adhesion~2 every torso of which is either $3$-connected, a circuit, or a cocircuit.
\item[\rm(ii)] This decomposition is unique in the sense that for any two such tree-decompositions, $(T,R)$ and~$(T',R')$ say, there is an isomorphism $v\mapsto v'$ between the trees such that $R_v = R'_{v'}$ for all~$v\in T$.
\end{enumerate}
\end{theorem}

Since $k$-separations of a matroid $M$ are also $k$-separations of its dual~$M^*$~\cite{BruhnWollanConInfMatroids},%
   \COMMENT{}
   a \td\ of $M$ is also one of~$M^*$, with the same adhesion. Moreover, the torsos corresponding to a given tree node are duals of each other:

\begin{theorem}\label{dual-main}
Every \td\ $(T, R)$ of a connected matroid $M$ is also a \td\ of its dual~$M^*$. If $(T,R)$ has uniform adhesion~2 for~$M$, it has uniform adhesion~2 also for $M^*$, and $(M_v)^* = (M^*)_v$ for all $v\in T$. In particular, $M$ and $M^*$ have the same unique irredundant tree-decomposition.
\end{theorem}

The notation we use in this paper is as follows.
Axiom systems for infinite matroids can be found in \cite{InfMatroidAxioms}. For other terminology we follow Oxley~\cite{OxleyBook}, or \cite{DiestelBook10noEE} for graphs. The letter $M$ always denotes a matroid. Its ground set, set of bases, and set of circuits will be denoted by $E(M)$, $\B(M)$ and $\C(M)$, respectively. Given $S \subseteq E(M)$, we let $M|S$ and $M/S$ denote the restriction of $M$ to $S$ and the contraction of $S$ in $M$, respectively, and write $S^\complement := E(M) \sm S$. The dual matroid of $M$ is denoted by~$M^*$.

%%%%%%%%%%%%%%%%%%%%%%%%%%%%%%%%%%%%%%%%%%%%%%%%%%%%%%%%%%%%

\section{Definitions, and outline of proof}\label{sec:outline}
In this section we give an outline of our proof of Theorem~\ref{main}, the details of which occupy the rest of the paper. In particular, we describe the construction of the \td\ whose existence is claimed in the theorem, and introduce the concepts needed to define it. We do not assume familiarity with the standard finite proof of Cunningham and Edmonds~\cite{CE}, but for readers familiar with that proof we emphasize the points where our (potentially) infinite setting requires a different approach. 

Throughout this section, let $M$ be a fixed connected matroid.

\noindent
\subsection{A tree of 2-separations} Two $k$-separations $(A, A^\complement)$ and $(B, B^\complement)$ of $M$ are said to be \emph{nested} if one of the four sets $A, A^\complement,B, B^\complement$ contains another which is equivalent to having at least one of the four sets $A\cap B$, $A^\complement \cap B$, $B^\complement \cap A$, $A^\complement \cap B^\complement$ empty. Two separations that are not nested are said to {\em cross\/}. A {\em good $k$-separation} is one that is nested with all other $k$-separations of $M$.

If $(T,R)$ is a \td\ of $M$ then the partitions $(S(e,v), S(e,w))$ of $E(M)$ that correspond to the edges $vw$ of $T$ are pairwise nested; indeed, the corresponding pair $(T_v,T_w)$ of subtrees of $T$ are nested. Hence, any \td\ of $M$ arises from choosing a {\sl suitable} subset of nested 2-separations out of the set of 2-separations of $M$. 
In particular, we show that the \td\ of Theorem~\ref{main} arises from choosing the set of all good 2-separations of $M$ (those are nested by definition).

For infinite matroids, this is not entirely trivial. One difficulty is that a decreasing chain $(A,A^\complement),\ldots,(B,B^\complement)$ of separations, one where $A\supsetneq\ldots \supsetneq B$, can now be infinite.
   \COMMENT{}
   If our claim that the good 2-separations correspond to the edges of a decomposition tree is true, then such infinite chains must not occur within the set of good 2-separations. For if $(A,A^\complement)$ and $(B,B^\complement)$ correspond to tree edges, then the tree will have only finitely many edges between these two, and hence the corresponding finite set of good 2-separations must be the only good 2-separations $(C,C^\complement)$ satisfying $A\supsetneq C\supsetneq B$ or $A\supsetneq C^\complement\supsetneq B$.

Since $B\ne\emptyset$, as $(B,B^\complement)$ is a 2-separation, the following lemma from Section~\ref{sec:inf-nested-main} implies that there are indeed no such infinite  chains of good 2-separations:%
  \COMMENT{}

\begin{proposition}\label{end}
Let $S_1 \supsetneq S_2 \supsetneq\ldots$ be an infinite sequence of subsets of $E(M)$ such that every partition $(S_i,S_i^\complement)$ is a good 2-separation of~$M$. Then $\bigcap_{i=1}^{\infty} S_i = \emptyset$. 
\end{proposition}

Another new difficulty in turning the set of good 2-separations into a \td\ is to define the parts corresponding to the nodes of the tree, indeed to define the tree itself. For $M$ finite, Cunningham and Edmonds obtain these parts and their torsos simultaneously, by splitting $M$ recursively along good 2-separations of the `current' matroid (not of~$M$) and adding a virtual element to each side in every split. When the recursion stops, the `current' matroids are the desired torsos. When $M$ is infinite, such a recursion would have to be transfinite,%
   \COMMENT{}
   and it is not clear how $M$ should induce matroids on the parts of the partitions (plus some virtual elements) that arise at limit steps.%
   \COMMENT{}
   We shall therefore define those matroids, the torsos of our \td, explicitly.

\subsection{Constructing the \td} We therefore define the decomposition tree explicitly, as follows (c.f. Section~\ref{sec:decomposition-main} for more details). In fact, we define a decomposition tree $T=T_\F$ for any symmetrical nested set%
   \COMMENT{}
   $\F$ of 2-separations of~$M$  (a set of separations is {\em symmetrical} if for every $(A,A^\complement)\in\F$, $(A^\complement, A)$ is also in $\F$; in the intended application, $\F$~will be the set of good 2-separations, and $T$ will be our decomposition tree.) Let us define the edges of $T$ first, and then its vertices, or~{\em nodes\/}. To define the edges, consider the partial ordering on $\F$ given by writing $(A,A^\complement) \le (B,B^\complement)$ whenever $A\subseteq B$. As the {\em edges\/} of $T$ we take the 2-separations in~$\F$ up to inversion:
 \begin{equation}\label{eq:ET} 
E(T_\F) := \big\{\{(A,A^\complement), (A^\complement, A)\} : (A,A^\complement)\in\F\big\}.
\end{equation}
   To define the nodes of~$T$, we call $(A,A^\complement)$ and $(B,B^\complement)$ {\em equivalent\/} if either $(A,A^\complement) = (B,B^\complement)$ or $(A^\complement,A)$ is a predecessor of $(B,B^\complement)$ in this ordering, i.e., if $A^\complement\subset B$ but there is no $(C,C^\complement)\in\F$ such that $A^\complement \subset C\subset B$. This is indeed an equivalence relation, and we take its classes as the nodes of~$T$:
 \begin{equation}\label{eq:VT} 
V(T_\F) := \big\{[(A,A^\complement)] : (A,A^\complement)\in\F\big\}.
\end{equation}
   We then let the edge $\{(A,A^\complement), (A^\complement, A)\}$ join the nodes $[(A,A^\complement)]$ and~$[(A^\complement, A)]$; these are distinct classes, since $(A,A^\complement)$ is not equivalent to~$(A^\complement, A)$.%
   \COMMENT{}
   Note that the degree of a node $v$ in $T$ is simply its cardinality, the number of good 2-separations in the equivalence class~$v$.

In order to turn $T$ into a decomposition tree of~$M$, we have to associate with every node $v$ of~$T$ a part $R_v\subseteq E(M)$ of the intended \td\ $(T,R)$, where $R = (R_v)_{v\in T}$. We do this by setting
 \begin{equation}\label{eq:Rv} 
R_v := \bigcap\, \{\,A\mid (A,A^\complement)\in v\,\}\,.
\end{equation}
When $v$ has degree at least~3 in~$T$, i.e.\ if $|v|\ge 3$, this set $R_v$ can be empty. We shall have to prove both that the graph $T$ thus defined is acyclic and that it is connected. Connectedness will follow from Proposition~\ref{end}. Indeed, we will prove the following.

\begin{lemma} \label{td}
When $\F$ is the set of all good 2-separations of~$M$, then $(T_\F,R)$, as defined above, is a \td\ of $M$ witnessing Theorem~\ref{main}\,(i).
\end{lemma}

\bigbreak\noindent
\subsection{Characterizing the torsos}
From the \td\ $(T,R)$ and its parts $R_v$ we define the {\em torsos\/} $M_v$ as in \eqref{eq:local-cir}. These torsos will be studied in detail in Section~\ref{sec:torso-is-matroid}. We later prove that torsos of the decomposition are 3-connected matroids, or circuits, or cocircuits. This will be done in Sections \ref{sec:decomposition-main} and~\ref{sec:torso-structure}, in two steps.

The first step will be to show that these torsos have no good 2-separations. Or equivalently, that any good 2-separation of a torso $M_v$ would give rise to a good 2-separation of~$M$ that splits~$R_v$, which by definition of~$R_v$ does not exist. This will be done in Section~\ref{sec:decomposition-main}. We also show there that our \td\ is irredundant. These properties, together with the remark following Lemma~\ref{prim-main} below, already imply its uniqueness as claimed in Theorem~\ref{main}\,(ii).%
   \COMMENT{}

As the second step, in Section~\ref{sec:torso-structure}, we show that the property of our torsos just established (that they have no good 2-separations) implies that they are 3-connected, circuits, or cocircuits:

\begin{lemma} \label{prim-main}
If $M$ has no good 2-separation, it is $3$-connected, a circuit, or a cocircuit.
\end{lemma} 

\noindent
The converse of this is easy: 3-connected matroids have no 2-separations at all, and any 2-separation of a circuit or cocircuit crosses another 2-separation.

\medbreak

Lemma~\ref{prim-main} is in turn proved in two steps; these are captured by the following two lemmas (which imply Lemma~\ref{prim-main}).

\begin{lemma}\label{prim}
If $M$ has no good 2-separation but is not 3-connected, then for every two elements $x,y$ the partition $(\{x,y\}, \{x,y\}^\complement)$ is a 2-separation.
\end{lemma}

\begin{lemma}\label{every}
If $M$ is such that for every two elements $x,y$ the partition $(\{x,y\}, \{x,y\}^\complement)$ is a 2-separation, then $M$ is a circuit or a cocircuit. 
\end{lemma}

\noindent
The converse of Lemma~\ref{every} is again easy.%
   \COMMENT{}

\medbreak

Lemmas \ref{prim} and~\ref{every} are proved in~\cite{CE} for finite matroids, but the proofs do not adapt to infinite matroids. In Section~\ref{sec:torso-structure} we provide alternative proofs.

%%%%%%%%%%%%%%%%%%%%%%%%%%%%%%%%%%%%%%%%%%%%%%%%%%%%%%%%%%%%

\section{Properties of 2-separations}
The purpose of this section is to study the properties of 2-separations in infinite matroids. This is necessary since the standard proofs for finite matroids~\cite{OxleyBook} do not always carry over.

Of the various axiom systems for infinite matroids established in~\cite{InfMatroidAxioms} we shall use the {\sl circuit axioms}: 
\begin{enumerate}
\item[(C1)] The empty set is not a circuit.
\item[(C2)] No circuit is a proper subset of another.
\item[(C3)] Whenever $X \subseteq C \in \C(M)$ and $\{C_x:x\in X\}$ is a family of circuits such that $x \in C_y \iff x=y$ for all $x,y \in X$, then for every $z \in C \sm (\bigcup_{x \in X}C_x)$ there exists a circuit $C'$ 
satisfying $z \in C' \subseteq   (C \cup (\bigcup_{x \in X}C_x)) \sm X$.
\item[(CM)] For every independent set $I$ (those sets not contained in any circuit $C$) and
any set $S$ containing $I$, there is a maximal independent subset of $S$ containing
$I$.
\end{enumerate} 
Axiom (C3) generalizes the traditional finite circuit elimination axiom,
and is referred to as the \emph{infinite circuit elimination axiom}.
The (CM) axiom is redundant for finite matroids.

The following notation will be frequently used. If $B$ is a base of a matroid $M$ and $e \in E(M) \sm B$, then we write $C_M(e,B)$ to denote the fundamental circuit of $e$ with respect to $B$; if the matroid $M$ is understood, then we omit the subscript.  

\subsection{New 2-separations from crossing 2-separations}
By definition, two~crossing $k$-separations define four nonempty disjoint sets. We refer to these sets as the \emph{quadrants} of these two crossing separations. In the next two lemmas, we shall see that for $k=2$, certain unions of these quadrants give rise to other 2-separations. 

Lemma~\ref{corner} below asserts that if a quadrant of two crossing 2-separations and its complement both have size at least 2, then they form a 2-separation as well. Before proving that lemma, we need the following definitions.

A function $f:E(M) \rightarrow \dR$ is called {\sl submodular} if 
\begin{equation}\label{eq:submodular} 
f(X) + f(Y) \ge f(X\cup Y) + f(X \cap Y) \ \mbox{for all $X,Y \subset E(M)$}.
\end{equation}

As mentioned in the Introduction, Bruhn and Wollan~\cite{BruhnWollanConInfMatroids} gave a rank free definition for the connectivity of a matroid. Given a matroid $M$ and two independent sets $I$ and $J$ of $M$, we follow~\cite{BruhnWollanConInfMatroids} in defining
$$
\del(I,J) = \min\{|F|: F \subseteq I \cup J, \; (I \cup J) \sm F \in \I(M) \}.
$$
The \emph{connectivity function} $\phi$ of $M$ is now defined as follows. Given $X \subseteq E(M)$, let $B_X$ and $B_{X^\complement}$ be two arbitrary bases of $M|X$ and $M|X^\complement$, respectively. Set
$$
\phi(X) = \del(B_X,B_{X^\complement}).
$$
The function $\phi$ is well defined~\cite[Lemma $14$]{BruhnWollanConInfMatroids} and submodular~\cite[Lemma $19$]{BruhnWollanConInfMatroids}.

\begin{lemma}\label{corner} \emph{(Corner lemma)}\\
Let $(S_1, S_1^\complement)$ and $(S_2, S_2^\complement)$ be 
two crossing 2-separations of a connected matroid $M$ such that $S_1\cap S_2$ and $(S_1 \cap S_2)^\complement$ both have size at least 2. Then, $(S_1 \cap S_2,\penalty-200\, (S_1 \cap S_2)^\complement)$ is a 2-separation. 
\end{lemma}

\begin{proof}
By assumption, $\phi(S_1) = \phi(S_2) =1$. Then, submodularity
of $\phi$ and the assumption that $M$ is connected yield
$$
1 \leq \phi(S_1 \cap S_2) \leq 2 - \phi(S_1 \cup S_2).
$$
As $M$ is connected, we have ${\phi(S_1 \cup S_2) \geq 1}$ and the lemma follows.
\end{proof}

The next lemma asserts that the union of two ``opposing'' quadrants of two crossing 2-separations and the complement of such a union form a 2-separation as well.  

\sloppy
\begin{lemma}\label{sim} \emph{(Symmetric difference lemma)}\\
If $(S_1, S_1^\complement)$ and $(S_2, S_2^\complement)$ are two crossing 2-separations of a connected matroid $M$, then $(S_1 \Delta S_2, (S_1 \Delta S_2)^\complement)$ is a 2-separation of $M$.  
\end{lemma}

\fussy
\begin{proof}
Put $X_1 = S_1 \cap S_2, X_2 = S_1^\complement \cap S_2, X_3 = S_1^\complement \cap S^\complement_2$, and $X_4 = S_1 \cap S^\complement_2$. As $S_1$ and $S_2$ cross these sets are all non-empty. We show that $(X_1 \cup X_3), (X_1 \cup X_3)^\complement)$ is a 2-separation of $M$. To that end, let $B_{(S_1 \Delta S_2)^\complement} \in \B(M|(X_1\cup X_3)) $ and choose $B_M \in \B(M)$ satisfying $B_{S_1 \Delta S_2} \subseteq B_M$.  Put $B_i = B_M \cap X_i$, for $1 \leq i \leq 4$.

Assume, towards a contradiction, that $(X_1\cup X_3, (X_1\cup X_3)^\complement)$ is not a 2-separation of $M$ so that $B_2\cup B_4$ is missing at least two elements from being a base of $M|(X_2\cup X_4)$. Let $e_2,e_4 \in X_2\cup X_4$ be two such (missing) elements.
  
\begin{equation}\label{sim.0}
\mbox{$e_4 \in X_4$ and $e_2 \in X_2$.}
\end{equation}
To see \eqref{sim.0}, suppose that $e_2,e_4 \in X_4$ (equivalently $X_2$). Then, $(X_4, (X_4)^\complement)$ is not a 2-separation of $M$ in contradiction to the corner lemma (Lemma~\ref{corner}).  Indeed, extend $B_{S_1 \Delta S_2}\cup B_2$ to a base of $M|(X_1\cup X_2\cup X_3)$ and $B_4\cup\{e_2,e_4\}$ to a base of $M|X_4$; then at least two elements must be removed from the union of these two bases in order to obtain a base of $M$.\\

Next, we show that 
\begin{equation}\label{sim.0.1}
\mbox{for each $j \in \{2,4\}$, $B_i\cup (B_j + e_j)$ is independent for at least one $i \in \{1,3\}$.}
\end{equation}
To see~\eqref{sim.0.1}, consider $e_2$. Since $B_2\cup B_4 \cup \{e_2,e_4\}$ is independent,  $B_2+e_2$ is independent. In addition, at least one of the sets $(B_2 + e_2) \cup B_1$ and $(B_2+e_2)\cup B_3$ is independent for otherwise $e_2$ has two distinct fundamental circuits with respect to $B_M$. 
A similar argument holds for $e_4$ and $B_4$ and thus~\eqref{sim.0.1} holds.

On the other hand, it holds that 
\begin{equation}\label{sim.1}
\mbox{there exists an $i \in \{1,3\}$ such that $B_i \cup (B_j+e_j)$ is dependent for each $j \in \{2,4\}$.}
\end{equation}
Suppose \eqref{sim.1} is false. By \eqref{sim.0.1}, we may suppose without loss of generality, that $B_1\cup(B_2+e_2)$ and $B_3\cup(B_4+e_4)$ are independent. Choose, now, $B_{1,2} \in \B(M|(X_1+X_2))$ and $B_{3,4} \in \B(M|(X_3+X_4))$ satisfying $B_1\cup(B_2+e_2) \subseteq B_{1,2}$ and  $B_3\cup(B_4+e_4) \subseteq B_{3,4}$, respectively. From $B_{1,2} \cup B_{3,4}$,  at least two elements must be removed in order to obtain a base of $M$; a contradiction to $(S_2,S_2^\complement)$ being a 2-separation. This proves \eqref{sim.1}.\\
 
Suppose then, without loss of generality that $i=1$ satisfies \eqref{sim.1}; that is, $B_1\cup (B_2 + e_2)$ and $B_1\cup(B_4+e_4)$ are dependent sets. Now $C(B_M,e_2) \subseteq B_1\cup(B_2 + e_2)$ and $C(B_M,e_4) \subseteq B_1\cup(B_4+e_4)$. Consequently, $B_2\cup B_3 \cup B_4 \cup \{e_2,e_4\}$ is independent. We may assume that $|X_1| \geq 2$; for if $X_1 = \{x\}$, then $e_2$ and $e_4$ are parallel, by applying circuit elimination on $C(B_M,e_2), C(B_M,e_4)$, and $x$. Choose bases $B'_1 \in \B(M|X_1)$  and $B_{2,3,4} \in \B(M|(X_2\cup X_3 \cup X_4))$ satisfying $B_1 \subseteq B'_1$ and $B_2\cup B_3 \cup B_4 \cup \{e_2,e_4\} \subseteq  B_{2,3,4}$, respectively. Then, $B'_1 \cup B_{2,3,4}$ indicates that $(X_1, (X_1)^\complement)$ is not a 2-separation of $M$, contradicting the corner lemma (Lemma~\ref{corner}).  

Hence $(S_1 \Delta S_2, (S_1 \Delta S_2)^\complement)$ is a 2-separation of $M$. 
%Moreover as $X_1, X_2, X_3, X_4$ are all nonempty, $|S_1\Delta S_2|, |(S_1 \Delta S_2)^\complement| \ge 2$ and the separation is proper as desired.
\end{proof}

\subsection{The limit of infinitely many nested $\boldsymbol k$-separations}\label{sec:limit}

In this section, we consider infinite sequences of nested $k$-separations. In particular,
our next lemma asserts that the limit of a nested sequence of $k$-separations is again an $\ell$-separation for some $\ell \leq k$, or a degenerate partition that cannot be a $k$-separation because one side is too small. This follows from~\cite[Lemma 20]{BruhnWollanConInfMatroids}; nevertheless, we include here a short proof for the convenience of the reader.

\begin{lemma}\label{nested}
Let $k \geq 1$ be an integer, and let $\sS = \{S_1 \supseteq S_2 \supseteq \cdots\}$ be a chain of nested subsets of~$E(M)$ with the property that $(S, S^\complement)$ is a $k$-separation of~$M$ for every $S\in\sS$. Then either $\big|\bigcap\sS\big| < k$, or $(\bigcap\sS, (\bigcap\sS)^\complement )$ is an $\ell$-separation of $M$ for some integer $\ell\le k$.
\end{lemma}

\begin{proof}
Let $S_\cap := \bigcap\sS$. Pick a basis $B$ of~$M|{S_\cap}$, extend it to a basis $B_M$ of~$M$, and extend $B_M\cap (S_\cap)^\complement$ to a basis $B'$ of $M|(S_\cap)^\complement$. If $|B'\setminus B_M| < k$, then $(S_\cap,(S_\cap)^\complement)$ is as desired. We may thus assume that $B'\setminus B_M$ contains set $Y$ of size $k$.

Since $Y\subseteq (S_\cap)^\complement$ is finite, there exists an $S\in\sS$ such that $Y \subseteq S^\complement$. Extend $B_M\cap S$ to a basis $B_S$ of~$M|S$, and $(B_M\cap S^\complement)\cup Y$ ($\subseteq B'$)
   \COMMENT{}
   to a basis $B_{S^\complement}$ of~$M|S^\complement$. Then $B_S\cup B_{S^\complement}$ exceeds its subset $B_M$ by at least the $k$-set~$Y$, contrary to our assumption that $(S,S^\complement)$ is a $k$-separation. 
\end{proof}

\noindent
It would be interesting to know whether Lemma~\ref{nested} always holds with $\ell = k$. 

\subsection{2-sums of infinite matroids}  
In this section, we consider the operation of taking a 
{\sl 2-sum} of two matroids. In the sequel, we shall use this operation 
to separate a connected matroid along a given
2-separation into two matroids, each a minor of the original matroid. 
The 2-sum operation, its properties, and typical uses are well known for finite matroids (see e.g.,~\cite{OxleyBook}); nevertheless, our infinite setting mandates that we study this operation and provide alternative proofs 
to some of its properties in a manner suitable for infinite matroids. 

Let $M_1$ and $M_2$ be two matroids having a single element $e$ in common, that is,  $E(M_1) \cap E(M_2) =\{e\}$. Let $\C_e$ denote the set comprised of the circuits of $M_i, i=1,2$ not containing $e$ together with the sets of the form 
$
(C_1-e) \cup (C_2 -e), 
$ 
whenever $e \in C_1 \cap C_2$ and $C_1 \in \C(M_1)$, and $C_2 \in \C(M_2)$. 
The set system $\C_e$ then defines a matroid as follows.

\begin{lemma}\label{2-sum-inf}
The set system $\C_e$ is the set of circuits of a matroid whose ground set is $E(M_1) \cup E(M_2) -e$.
\end{lemma}

\noindent
The matroid defined in Lemma~\ref{2-sum-inf} is called the 2-\emph{sum} of $M_1$ and $M_2$, and is denoted by $M_1 \oplus_2 M_2$. 
In what follows we prove Lemma~\ref{2-sum-inf} in a manner suitable for infinite matroids. To that end, we prove the following. %(see e. g., \cite{oxley}).

\begin{lemma}\label{2-sum}
If $(S, S^\complement)$ is a 2-separation of $M$, then there exist two matroids $M_1,M_2$ such that $E(M_1) = S + e$ and $E(M_2)=S^\complement+e$, where $e \notin E(M)$ so that $M = M_1 \oplus_2 M_2$. Moreover, $M_i$ is isomorphic to a minor of $M$, for $i=1,2$.
\end{lemma}

Lemma~\ref{2-sum} is a corollary of Lemma~\ref{localmatroid1} (stated below) and so we postpone the proof of the former. The latter is one of the main results of Section~\ref{sec:torso-is-matroid}.  In the remainder of this section we establish what we call the {\sl infinite switching lemma}:

\begin{lemma}\label{infswitching}{\rm (Infinite switching lemma)}\\
Let $\{S_i: i\in I\}$ be a set of disjoint subsets of $E(M)$ where $(S_i, S_i^\complement)$ is a 2-separation of $M$ for every $i\in I$. If 
\begin{enumerate}
\item [{\rm (1)}] $C_1$ and $C_2$ are circuits each crossing $(S_i, S_i^\complement)$ for all $i$, and  
\item [{\rm (2)}] $C_2$ meets $\left( \bigcup_{i \in I}S_i \right)^\complement$ if $C_1$ does,
\end{enumerate}
then 
$
(C_1 \cap \bigcup_{i\in I} S_i) \cup (C_2 \cap (\bigcup_{i\in I}S_i)^\complement)
$
is a circuit.
\end{lemma}

The infinite switching lemma (Lemma~\ref{infswitching}) will be used repeatedly throughout and in particular in the proof of Lemma~\ref{localmatroid1}.  For future reference,  it will be convenient for us to mention the following special case of the infinite switching lemma, to which we refer simply as the {\sl switching lemma}:

\begin{lemma}\label{switching}{\rm (Switching lemma)}\\
If $C_1$ and $C_2$ are circuits of $M$ crossing a 2-separation $(S,S^\complement)$ of $M$, then $(C_1 \cap S) \cup (C_2 \cap S^\complement)$ is a circuit.
\end{lemma}

The switching lemma appears in~\cite{OxleyBook}; the proof proposed in~\cite{OxleyBook} for this lemma does not fit for infinite matroids. Indeed, in order to have it hold for infinite matroids, one seems to need the infinite circuit elimination axiom, i.e., (C3).  

The following lemma found in~\cite{OxleyBook} facilitates our proof of the infinite switching lemma (Lemma~\ref{infswitching}). We include a proof of Lemma~\ref{improper} for completeness.
A circuit $C$ of $M$ is said to \emph{cross} a 2-separation $(S, S^\complement)$ of $M$ if $C$ meets both $S$ and $S^\complement$.

\begin{lemma}\label{improper}
If $C_1$ and $C_2$ are circuits of $M$ both crossing a 2-separation $(S,S^\complement)$ of $M$, then $C_1 \cap S$ is not a proper subset of $C_2 \cap S$.
\end{lemma}

\begin{proof}
Assume, to the contrary, that $C_1 \cap S \subsetneq C_2 \cap S$ and let $e_1 \in C_1 \cap S$ and $e_2 \in (C_2 \sm C_1) \cap S$. Choose $B_S \in \B(M|S)$ and $B_{S^\complement} \in \B(M|S^\complement)$ satisfying $C_1 \cap S \subseteq B_S$ and $C_1 \cap S^\complement \subseteq B_{S^\complement}$, respectively.
Since $S$ is a 2-separation, $Z= (B_S\cup B_{S^\complement}) \sm \{e_1,e_2\}$ is independent. 
Observe now that $Z$ is spanning; indeed,  $E(M) - e_1$ is  spanned by $Z$, and as $C_1-e_1 \subseteq Z$, the element $e_1$ is spanned by $Z$ as well. This contradicts the assumption that $S$ is a 2-separation.
\end{proof}

We are now ready to prove the infinite switching lemma (Lemma~\ref{infswitching}).

\begin{proofof}{Infinite Switching Lemma}
Put $C=(C_1 \cap \bigcup_{i\in I} S_i) \cup (C_2 \cap (\bigcup_{i\in I}S_i)^\complement)$. We prove that $C$ is a circuit. 
To that end we first prove that 
\begin{equation}\label{cisdependent}
\text{$C$ forms a dependent set.}
\end{equation}

\begin{innerproofof}{\eqref{cisdependent}}
Assume towards a contradiction that  $C$ is independent.  Extend $C$ to a base $B_M$ of $M$ and set $X_i = B_M \cap S_i$.

\noindent
Either:

(a) there exists an $i$ such that $B_M\cap S_i^\complement$ is a base of $M| S_i^\complement$, or 

(b) there is no such $i$. 

\noindent
Consider case (a); let $z$ be an element of $X_i$, set $V=C_1\setminus B_M$, and for each $e \in V$ let $C_e$ denote its fundamental circuit into $B_M\cap S_i^\complement$. Then, by the infinite circuit elimination axiom (C3) applied to $C_1$, $z$, $V$, and $\{C_e: e\in V\}$, there exists a circuit in $C_1 \cup \bigcup_{e\in V} C_e \setminus V \subseteq B_M$; a contradiction.

We may now assume that case (b) holds; that is, 
$B_M \cap S_i^\complement$ is not a base of $M|S_i^\complement$ for any $i$.
This together with the assumption that $(S_i,S_i^\complement)$ is a 2-separation of $M$ for every $i$ implies that 
\begin{equation}
\mbox{$X_i = B_M \cap S_i$ is a base of $M|S_i$ for every $i$.}
\end{equation} 
We arrive at a contradiction in this case as follows. 

As $C_2$ is not contained in $B_M$, we may assume, without loss of generality, that 
\begin{equation}
\mbox{$Y = (C_2 \cap S_1) \sm X_1$ is nonempty.}
\end{equation} 
Set 
\begin{equation}
\mbox{$V = C_2 \sm (B_M \cup Y)$},
\end{equation} 
and note that $V \subseteq \bigcup_{i>1} (C_2 \cap S_i) \sm X_i$. 
We may assume that 
\begin{equation}
\mbox{$V$ is nonempty.}
\end{equation} 
Indeed, for otherwise, choose a $y \in C_2 \cap X_i$ for some $i \not=1$, such an element $y$ exists as $C_2$ crosses $(S_i,S_i^\complement)$. Applying the infinite circuit elimination axiom to $C_2$, $y$, $Y$, and 
$\{C_{M|S_i}(e,X_i): e \in Y \}$, yields a circuit contained in $B_M$ which is a contradiction. 

For each $e \in V$, there exists an $i_e$ such that $e \in S_{i_e}$. Let $C_e$ denote the fundamental circuit of $e$ into $X_{i_e}$ in $M|S_{i_e}$. In addition, choose an element $z \in Y$. Then, by the infinite circuit elimination axiom applied to $C_2$, $z$, $V$, and $\{C_e: e\in V\}$, there exists a circuit $C_3$ contained in $\left(C_2 \cup \bigcup_{e \in V} C_e\right)\sm V$ so that 
\begin{equation}
\mbox{$C_3 \sm Y \subseteq B_M$.}
\end{equation} 
Observe that if $C_3$ does not meet $\left(\bigcup_{e \in V} C_e\right) \sm V$, then $C_3 \subset C_2$ which is a contradiction. Consequently, 
\begin{equation}
\mbox{$C_3$ crosses $(S_1,S_1^\complement)$.} 
\end{equation}
Then, the infinite circuit elimination axiom applied to $C_3$, an element of $C_3 \cap S_1^\complement$, the set $Y$, and the set $\{C_{M|S_1}(e,X_1):e \in Y \}$, yields a circuit contained in $B_M$ which is a contradiction. 
\end{innerproofof}

To conclude, we show that in fact 
\begin{equation}\label{ciscircuit}
\text{$C$ is a minimal dependent set; i.e., a circuit}.
\end{equation}

\begin{innerproofof}{\eqref{ciscircuit}}
$C$ is dependent and thus it suffices to consider the minimality of $C$. Assume then towards contradiction that $C$ is not a minimal dependent set and let $C_3$ be a circuit contained in $C$. We show that $C$ coincides with $C_3$. 

As $C_3$ is not properly contained in either $C_1$ or $C_2$ it follows that $C_3$ meets $C_2 \cap \left(\bigcup_i S_i \right)^\complement$ and also meets $C_1 \cap S_i$ for at least one $i \in I$. In particular, there exists an $i \in I$ such that $C_3$ crosses $(S_i,S_i^\complement)$. 

Let $I' \subseteq I$ be those indices $i \in I$ such that $C_3$ crosses $(S_i,S_i^\complement)$. By Lemma~\ref{improper}, we have that $C_3 \cap S_i = C_1 \cap S_i$ for each $i \in I'$. 
Next, consider $C' = (C_2\cap \bigcup_{i\in I'} S_i) \cup (C_3 \cap (\bigcup_{i\in I'}S_i)^\complement)$. If $C'$ is a proper subset of $C_2$ and thus independent, then we are in the previous case with $I$ replaced with $I'$. The set $C'$ is not a proper subset of $C_2$ provided that 
\begin{equation}\label{eq:coincide}
C_3 \cap \Big(\bigcup_{i\in I'}S_i\Big)^\complement = C_2 \cap \Big(\bigcup_{i\in I'}S_i\Big)^\complement.
\end{equation} 
This has two implications. First, it holds that  $C_3 \cap \left(\bigcup_{i\in I}S_i\right)^\complement = C_2 \cap \left(\bigcup_{i\in I}S_i\right)^\complement$. Second, it implies that $I' = I$. Indeed, if $I' \subset I$, then \eqref{eq:coincide} implies that $C_3 \cap S_i = C_2 \cap S_i$ for each $i \in I \sm I'$ implying that $C_3$ does cross $(S_i,S_i^\complement)$ for an $i \in I \sm I'$ which is a contradiction to the definition of $I'$. 

We have shown that $C$ and $C_3$ coincide and so~\eqref{ciscircuit} is established 
\end{innerproofof}

The lemma now follows. 
\end{proofof}

%%%%%%%%%%%%%%%%%%%%%%%%%%%%%%%%%%%%%%%%%%%%%%%%%%%%%%%%%%%%

\section{Localizations}\label{sec:torso-is-matroid}
In this section, we study a notion to which we refer as a {\sl localization}; this is essentially a minor of a connected matroid $M$ that has been ``isolated'' or ``pointed at'' by a certain set of 2-separations. In particular, torsos (as defined in the previous sections) are localizations with the ``localizing'' 2-separations chosen all to be good. 

Throughout this section, $\U =\{X_i : i\in I\}$ is a set of disjoint subsets of a connected matroid $M$  where $(X_i, X_i^\complement)$ is a 2-separation of $M$ for all $i$. Roughly speaking, a localization will be a matroid obtained by essentially contracting $M$ onto the complement of $\bigcup X_i$, and then adding certain ``virtual'' elements instead of the members of $\U$. To prove that the resulting object is, in fact, a matroid, we will show that the set comprised of circuits of $M$ that are not contained in any member of $\U$ gives rise to a set system that in turn defines the set of circuits of a matroid. We now make this precise.

Let us write 
\begin{equation}
R(\U)=E(M)\setminus \bigcup_{i\in I} X_i
\end{equation}
to denote the elements of $M$ not contained in any member of $U$. These elements are called the \emph{real} elements of the intended matroid. The ground set of the intended matroid is given by  
\begin{equation}
E(\U)= \{e_i: X_i\in \U\} \cup R(\U),
\end{equation}
where the elements $e_i$ are distinct and all are not in $E(M)$; 
we call these elements \emph{virtual}. 
Next, given a subset $Y \subseteq E(M)$, we set  
\begin{equation}
\phi_{\U}(Y) = \{e_i: Y\cap X_i \ne \emptyset\} \cup \left( Y\cap R(\U) \right),
\end{equation} 
and say that $Y$ \emph{induces} $\phi_{\U}(Y)$. So $\phi_{\U}$ is simply a mapping from the subsets of $E(M)$ to the subsets of $E(\U)$.
Finally, let $\C_{\U}(M)$ denote the circuits of $M$ not contained in any $X_i$, that is, 
\begin{equation}
\C_{\U}(M) = \{C\in \C(M)| \nexists i \; \mbox{such that} \; C\subseteq X_i\}.
\end{equation} 

The following is the first main result of this section.  

\begin{lemma}\label{localmatroid1}
The set
$
\C(\U) = \{ \phi_{\U}(C) : C\in \C_{\U}(M)\}
$
is the set of circuits of a matroid whose ground set is $E(\U)$.  
\end{lemma}

\begin{definition}
The matroid of Lemma~\ref{localmatroid1} is called the \emph{localization of $M$ at $\U$}, and is denoted by $M_{\U}$.
\end{definition}

The second main result of this section is Lemma~\ref{localgoodseps} (see below). 
This lemma essentially asserts that a good 2-separation of a localization $M_{\U}$ of $M$ gives rise to a good 2-separation of $M$. We shall use this lemma in Section~\ref{sec:decomposition-main} to argue that torsos admit no good 2-separations. We postpone discussion of this lemma until later sections.

This section is organized as follows. Sections~\ref{sec:C1-3} and~\ref{sec:CM} are dedicated to the proof of Lemma~\ref{localmatroid1}. In these sections we verify that $\C(\U)$ satisfies the circuit axioms and thus defines a matroid. In Section~\ref{sec:local-props}, we prove Lemma~\ref{localgoodseps}.

\subsection{The axioms (C1)--(C3) for localizations}\label{sec:C1-3}
We begin by verifying that $\C(\U)$ satisfies axiom (C1). 

\begin{claim}\label{C1}
The empty set is not in $\C(\U)$ so that $\C(\U)$ satisfies (C1).  
\end{claim}

\begin{proof}
The empty set is not in $\C(M)$, by (C1), and the image of a nonempty set under $\phi_{\U}$ is a nonempty set. 
\end{proof}

To verify (C2), we observe the following. 

\begin{claim}\label{C2}
If $C_1,C_2 \in \C(\U)$, then $C_1$ is not a proper subset of $C_2$.
\end{claim}

\begin{proof}
Suppose that $C_1$ is a proper subset of $C_2$, and let $C_1',C_2' \in \C_{\U}(M)$ be circuits satisfying $C_1=\phi(C_1')$ and $C_2=\phi(C_2')$. By the infinite switching lemma (Lemma~\ref{infswitching}), we may assume that $C_1'\cap X_i = C_2' \cap X_i$ for all $i$ where $C_1'\cap X_i \ne \emptyset$. It follows now that $C_1'$ is a proper subset of $C_2'$; a contradiction to axiom (C2) for $M$.
\end{proof}

Next, we consider the axiom (C3).

\begin{claim}\label{C3}
$\C(\U)$ satisfies the infinite circuit elimination axiom (C3).
\end{claim}

\begin{proof}
Let $C_{\U}\in \C(\U)$, let $z_{\U}\in C_{\U}$, and let $V_{\U}\subseteq C_{\U}$ such that $z_{\U}\not\in V_{\U}$. Suppose now that $\{C_v: v\in V_{\U}\}$ is a subset of $\C(\U)$ satisfying the property stated in axiom (C3) that $u \in C_v$ if and only if $u=v$ for all $u,v \in V_{\U}$, and that $z_{\U} \notin \bigcup_{v \in V_{\U}}C_v$. We prove that there is a member of $\C(\U)$ contained in $\left( C_U\cup \bigcup_{v\in V_{\U}} C_v \right) \setminus V_{\U}$; moreover, this member of $\C(\U)$ contains $z_{\U}$.  

Let $C_M$ be a circuit of $M$ satisfying $C_{\U}=\phi(C_M)$. For a set $X_i \in \U$ with the property that its virtual element $e_i \in C_{\U}$, set $D_i=C_M \cap X_i$, and let $d_i \in D_i$. For each $v\in V_{\U}$, let $C_v'$ be a circuit of $M$ satisfying $C_v=\phi(C_v')$. By the infinite switching lemma (Lemma~\ref{infswitching}), we may assume that, for all $v\in V_{\U}$ and $i\in I$ it holds that if $e_i \in C_v\cap C_{\U}$, then $C_v'\cap X_i = D_i$.
 
Set $z_M = d_i$ if $z_{\U}=e_i$ for some $i$, and set $z_M=z_{\U}$ otherwise. In a similar manner, for $v\in V_{\U}$, set $v'= d_i$ if $v=e_{X_i}$, and set $v'=v$ otherwise. Put $V_M= \bigcup_{v\in V_{\U}} v'$. Now, by the infinite circuit elimination axiom (C3) applied to $C_M$, $z_M$, $V_M$, and $\{C_v': v\in V_M\}$, there exists a circuit $C_M'$ of $M$ in $C_M \bigcup_{v'\in V_M} C_v' \setminus V_M$ such that $C'_{\U}=\phi(C_M')$ is the desired set.
\end{proof}

\subsection{The (CM) axiom for localizations}\label{sec:CM}

The aim of this section is to prove Claim~\ref{CM} (stated below) asserting that the set system $\C(\U)$ satisfies axiom (CM), and consequently conclude our proof of Lemma~\ref{localmatroid1}, thus establishing that $M_{\U}$ is a matroid. To that end, it will be convenient to have a description of the independent sets and, in particular, the bases of this intended matroid. We consider this next. 

Let $\I(\U)$ be the set system consisting of all subsets of $E(\U)$ not containing a member of $\C(\U)$. The following describes $\I(\U)$. Given an independent set $I \in \I(M)$ it is not hard to show that the union of the set $I \cap R(\U)$ with the set $\{e_i : I \cap X_i \in \B(M|X_i)\}$ is a member of $\I(\U)$. In fact, all members of 
$\I(\U)$ are of this form.
\begin{equation}\label{localindep}
\I(\U)= \{(I \cap R(\U)) \cup \{e_i : I \cap X_i \in \B(M|X_i)\}: I\in \I(M)\}.
\end{equation}

\begin{proofof}{\eqref{localindep}} 
Let $I_{\U} \in \I(\U)$, choose a base $B_i$ of $M|X_i$ for each virtual element $e_i \in I_{\U}$,
and set $I_M = (I_{\U} \cap R(\U)) \cup \left( \bigcup_{e_i\in I_{\U}} B_i\right)$. We show that $I_M$ is independent in $M$. Suppose not, and let $C_M$ be a circuit of $M$ contained in $I_M$. 
Clearly, $C_M$ is not contained in any member $X_i$ of $\U$, for otherwise $C_M \subseteq B_i$. 
Hence, $C_M\in C_{\U}(M)$ and induces a set $C_{\U}\in \C(\U)$ satisfying $C_{\U}\subset I_{\U}$; a contradiction.

Conversely, let $I_M$ be an independent set in $M$, and consider $I_{\U} = (I_M \cap R(\U)) \cup \{e_i: I_M \cap X_i \in \B(M|X_i)\}$. We show that $I_{\U}\in \I(\U)$. Suppose not. Then, $I_{\U}$ contains a member $C_{\U}$ of $\C(\U)$. Choose $C_M\in C_{\U}(M)$ such that $C_M$ induces $C_{\U}$, and let $V=C_M\setminus I_M$. Observe that $C_M\cap R(\U) \subseteq I_M \cap R(\U)$. Hence, if $v\in V$, then $v\in X_i$ for some $i$ such that $I_M\cap X_i \in B(M|X_i)$. Consequently we define $C_v = C(v, I_M\cap X_i)$ for all $v\in V$. Applying the infinite circuit elimination, if necessary with two different $z$'s, to $C_M$ using $V$ and $\{C_v: v\in V\}$, we obtain a circuit $C_M'$ in $I_M$, a contradiction.
\end{proofof}

Next, we determine the bases of a localization. 
Let $\B(\U)$ denote the set system consisting of the maximal members of $\I(\U)$.

\begin{lemma}\label{localbases}
%\label{local-bases.1}
$\B(\U)= \{(B \cap R(\U)) \cup \{e_i : B \cap X_i \in \B(M|X_i)\}: B\in \B(M)\}$.
\end{lemma}

\begin{proof}
To prove the lemma, we shall use Subclaim~\ref{nonbasenocircuit} stated below. To prove the latter, we require the following. 

\setcounter{thesubclaim}{0}
\begin{subclaim}\label{baseuniquecircuit}
Let $(S,S^\complement)$ be a 2-separation of $M$, and let $B_S$ be a base of $M|S$. If $C_1,C_2 \in \C(M)$ satisfy $C_1\cap S, C_2\cap S \subseteq B_S$, then $C_1\cap S = C_2\cap S$.
\end{subclaim}

\begin{innerproofof}{Subclaim 1}
Suppose the claim is false and let $z \in (C_1 \cap B_S) \sm (C_2 \cap B_S)$. By the switching lemma (Lemma~\ref{switching}), the set $C_3 = (C_1 \cap B_S) \cup (C_2\cap S^\complement)$ is a circuit of $M$. 
By the circuit elimination axiom applied to $C_3$, $z$, $C_2$, and an element $w \in C_2 \cap S^\complement$, there exists a circuit $C_4$ contained in $C_3 \cup C_2 - w$ such that $z \in C_4$. This circuit cannot cross $(S,S^\complement)$, for if so then $C_4 \cap S^\complement$ is properly contained in $C_2 \cap S^\complement$, a contradiction to Lemma~\ref{improper}. Then, (since $z \in C_4$) we have that $C_4 \subseteq S$ implying that $C_4 \subseteq B_S$, which is a contradiction as well.  
\end{innerproofof}

\begin{subclaim}\label{nonbasenocircuit}
Let $(S,S^\complement)$ be a 2-separation of $M$. Suppose that $B$ is a base of $M$ such that $B\cap S$ is not a base of $M|S$. Then there does not exist a circuit $C$ of $M$ such that $C \cap S\subseteq B$.
\end{subclaim}

\begin{innerproofof}{Subclaim 2} 
Suppose there does exist such a circuit $C$ of $M$. Let $B_S$ be a base of $S$ containing $B\cap S$. Let $f\in B_S\setminus B$. Now $C(f,B)\cap S$ and $C\cap S$ are both contained in  $B_S$, yet $C(f,B)\cap S\ne C\cap S$, contradicting Subclaim~\ref{baseuniquecircuit}.
\end{innerproofof}

Given Subclaim~\ref{nonbasenocircuit}, we proceed to proving Lemma~\ref{localbases} as follows.
Let $B_{\U}$ be a maximal element of $\I(\U)$. By \eqref{localindep}, there exists an independent set $I_M$ of $M$ such that $B_{\U} = (I_M \cap R(\U)) \cup \{e_i : I_M \cap X_i \in \B(M|X_i)\}$. Extend $I_M$ to a base $B_M$ of $M$, and set $B'_{\U} = (B_M \cap R(\U)) \cup \{e_i : B_M \cap X_i \in \B(M|X_i)\}$. Then, $B'_{\U}$ is in $\I(\U)$, by \eqref{localindep}. As $I_M \subseteq B_M$, we have that $B_{\U}\subseteq B'_{\U}$. As $B_{\U}$ is maximal, the equality $B_{\U}=B'_{\U}$ holds. Thus, $B_{\U} = (B_M \cap R(\U)) \cup \{e_i : B_M \cap X_i \in \B(M|X_i)\}$ as desired.

For the converse direction, let $B_M$ be a base of $M$. Let $I_{\U}= (B_M \cap R(\U)) \cup \{e_i : B_M \cap X_i \in \B(M|X_i)\}$. We show that $I_{\U}$ is in $\B(\U)$. Clearly, $I_{\U}\in \I(\U)$ since 
$B_M$ is independent in $M$. To show that $I_{\U}$ is maximal in $\I(U)$ it is sufficient to prove that for all $e\in E(\U) \setminus I_{\U}$, the set $I_{\U}+e$ contains a member of $\C(\U)$. 

Let $e\in E(\U)\setminus I_{\U}$; suppose, first, that $e\in R(\U)$, and let $C_M=C_M(e,B_M)$. By Subclaim~\ref{nonbasenocircuit}, $C_M\cap X_i = \emptyset$ if $e_i\not\in B_{\U}$. So $C_M$ induces a set $C_{\U}\in \C(\U)$ such that $C_{\U}\subseteq I_{\U}+e$. So $e$ is spanned by $I_{\U}$.

Suppose, second, that $e$ is some virtual element $e_i$; so $B_M\cap X_i$ is not a base of $M|X_i$. Choose $f\in E(M)$ such that $B\cap X_i + f\in \B(M|X_i)$. Let $C_M=C(f,B_M)$. Now $C_M \in C_{\U}(M)$. Hence $C_M$ induces a set $C_{\U}\in \C(\U)$. By Subclaim~\ref{nonbasenocircuit}, $C_M\cap X_i = \emptyset$, if $e_i\not\in I_{\U}+e$. Hence, $C_{\U}\subseteq I_{\U}+e$. Thus, $I_{\U}$~spans~$e$.
\end{proof}

We conclude this section by proving that $\C(\U)$ satisfies the (CM) axiom, and thus completing our proof of Lemma~\ref{localmatroid1}. First, let us state the following lemma that is easy to verify and which will facilitate our proof of Claim~\ref{CM} (below).

\begin{lemma} \label{restriction}
Let $(S, S^\complement)$ be a 2-separation of $M$, and let $X\subseteq E(M)$ such that $|S\cap X|, |S\cap X^\complement| \geq 2$. Then $(S\cap X, S\cap X^\complement)$ is a 2-separation of $M|X$.
\end{lemma}
%
%\noindent
%That is, Lemma~\ref{restriction} asserts that 2-separations of a matroid $M$ are inherited by restrictions $M|X$ where $X \subseteq E(M)$.

Claim~\ref{CM} is stated and proved next. 
  
\begin{claim}\label{CM}
$\C(\U)$ satisfies (CM).
\end{claim}

\begin{proof}
Let $A_{\U}$ be a subset of $E(\U)$ and let $I_{\U}$ be a member of $\I(\U)$ contained in $A_{\U}$. 
We are to show that $I_{\U}$ is contained in a maximal member of $\{ I\in \I(\U): I\subseteq A_{\U}\}$. 
To that end, let $A_M \subseteq E(M)$ be given by $A_M = \{x: \phi_{\U}(x)\in A_{\U}\}$; this set consists of $A_{\U} \cap R(\U)$ together with the members $X_i$ for each virtual element $e_i$ present in $A_{\U}$. 
Finally, let $I_M \in \I(M)$ be the independent set of $M$ giving rise to $I_{\U}$ per \eqref{localindep}.   

Extend $I_M \cap A_M$ to a base $B_{A_M}$ of $M|A_M$, and put 
$$
B_{A_{\U}} = (B_{A_M} \cap A_{\U}) \cup \{e_i: B_{A_M} \cap X_i \in \B(M|X_i)\}.
$$
We show that $B_{A_{\U}}$ is the required maximal member of $\{ I\in \I(\U): I\subseteq A_{\U}\}$. To see that $B_{A_{\U}} \in \I(\U)$ note that $I_{\U} \subseteq B_{A_{\U}} \subseteq A_{\U}$. Next, extend $B_{A_M}$ to a base $B_M$ of $M$ so that 
$B_{A_{\U}} \in \I(\U)$ follows, by~\eqref{localindep}.

To show that $B_{A_{\U}}$ is maximal in the required sense, it is sufficient to show that $B_{A_{\U}} +e$ contains a member of $\C(\U)$ whenever $e \in A_{\U} \sm B_{A_{\U}}$. If $e$ is real, i.e., $e \in R(\U)$, the circuit $C_{M|A_M}(e,B_{A_M})$ gives rise to a member of $\C(\U)$ in $B_{A_{\U}}+e$. 

Suppose then that $e$ is some virtual element $e_i$. Consider $(X_i,A_M\sm X_i)$. By definition, $|X_i| \geq 2$. As $I_{\U}$ is nonempty and does not contain $e_i$, by assumption, we have that $|A_M \sm X_i| \geq 1$. We may, in fact, assume that $|A_M \sm X_i| \geq 2$, for otherwise $1 =|I_{\U}| \leq |A_{\U}| \leq 2$ and the claim is trivially true. Consequently, $(X_i,A_M\sm X_i)$ is a 2-separation of $M|A_M$, by Lemma~\ref{restriction}. Now, if $M|A_M$ contains a circuit that crosses $(X_i,A_M\sm X_i)$, then such a circuit gives rise to a member of $\C(\U)$ that is contained in $B_{A_{\U}}+e_i$. The complementary case that no circuit of $M|A_M$ crosses $(X_i,A_M\sm X_i)$ does not occur. Indeed under such an assumption $M|A_M$ is disconnected. In addition, recall that $e_i \notin B_{A_{\U}}$ since $B_{A_M} \cap X_i$ is not a base of $M|X_i$. These two facts imply that $B_{A_M}$ can be extended in $X_i$ without picking up a circuit of $M|A_M$ contradicting the assumption that it is a base of $M|A_M$. 
\end{proof}

\subsection{Good 2-separations of localizations}\label{sec:local-props}
The purpose of this section is to ``relate" 2-separations of (the entire matroid) $M$ with 2-separations of a given localisation $M_{\U}$. This is done in Lemma~\ref{local2seps}. As good 2-separations are of prime interest to us we shall also "relate" good 2-separations of $M$ with those of $M_{\U}$. This is done in Lemma~\ref{localgoodseps}; the main result of this section. We now make this precise.

Suppose $(S,S^\complement)$ is a 2-separation of $M$. 
%In Lemma~\ref{local2seps} (below), we shall see that $(\phi_{\U}(S),\phi_{\U}(S^\complement)$ is a 2-separation of $M_{\U}$; this maps 2-separations of $M$ to 2-separations of $M_{\U}$. In order to map 2-separations of $M_{\U}$ to 2-separations of $M$, we 
Define
$$
\phi_{\U}^{-1}(Z) := \{y \in E(M)| \phi_{\U}(y) \in Z\},
$$
where $Z \subseteq E(M_{\U})$. In particular, let us remark that if $e_i \in E(M_{\U})$ is the virtual element representing the member $X_i$ of $\U$ in $M_{\U}$, then $\phi^{-1}_{\U}(e_i) = X_i$. 
Below we prove the following, asserting a correspondence between the 2-separations of $M$ and those of its localization $M_{\U}$. Prior to this let us set the notation that $S^\complement$ means $E(M_{\U}) \sm S$ whenever $S \subseteq E(M_{\U})$.   

\begin{lemma}\label{localgoodseps}
Let $(S, S^\complement)$ be a 2-separation of $M_{\U}$. Then, $(S,S^\complement)$ is a good 2-separation of $M$ if and only if $(\phi_{\U}^{-1}(S), \phi_{\U}^{-1}(S^\complement))$ is a good 2-separation of $M$.
\end{lemma}

We prepare for the proof of this lemma. For a set $A \subseteq E(M_{\U})$,  the set system 
$$
\U_A = \{X_i \in \U | \; \phi_{\U}(X_i) \in A\}
$$ 
consists of those members $X_i$ of $\U$ that are mapped to some member of $A$ by $\phi_{\U}$; so that if $A$ contains no virtual elements, then $\U_A$ is empty. Consider now the matroid $M|\phi_{\U}^{-1}(A)$; the ground set of which is 
$$
\bigcup_{X \in \U_A} X \cup (A \cap R(\U)).
$$ 
By Lemma~\ref{restriction}, a pair $(X,\phi_{\U}^{-1}(A) \sm X)$ with $X \in \U_A$ and satisfying $|\phi_{\U}^{-1}(A) \sm X| \geq 2$ forms a 2-separation of $M|\phi_{\U}^{-1}(A)$.  Consequently, 
$(M|\phi_{\U}^{-1}(A))_{\U_A}$ is a localization of $M|\phi_{\U}^{-1}(A)$ at $\U_A$ provided $|\phi_{\U}^{-1}(A) \sm X| \geq 2$ holds for every $X \in \U_A$. The next claim asserts that this localization is simply the matroid $M_{\U}|A$.

\begin{claim}\label{localrestriction}
Let $A\subseteq M_{\U}$ such that $(M|\phi_{\U}^{-1}(A))_{\U_A}$ is a localization (and hence a matroid).
Then, $M_{\U}|A = (M|\phi_{\U}^{-1}(A))_{\U_A}$.
\end{claim} 

\begin{proof}
These two matroids have the same ground set; it suffices now to show that they have the same circuits. 
Let, then, $C$ be a circuit of $M_{\U}|A$, and let $C_M$ be a circuit of $M$ satisfying $\phi_{\U}(C_M)=C$. Hence, $C_M \subseteq \phi_{\U}^{-1}(A)$, so $C_M$ is a circuit of $M|\phi_{\U}^{-1}(A)$. Thus $C' = \phi_{\U_A}(C_M)$ is a circuit of $(M|\phi_{\U}^{-1}(A))_{\U_A}$. But $C' = C$. 

Now let $C$ be a circuit of $(M|\phi_{\U}^{-1}(A))_{\U_A}$, and let $C_M$ be a circuit of  $M|\phi_U^{-1}(A)$ satisfying $\phi_{\U_A}(C_M) = C$. Then, $C_M$ is a circuit of $M$ such that $C_M \subseteq  \phi_{\U}^{-1}(A)$. Thus, $C' = \phi_{\U}(C_M)$ is a circuit of $M$ such that $C' \subseteq A$. But $C'=C$.
\end{proof}

A characterization of 2-separations is now available. 

\begin{lemma}\label{local2seps}
Let $S \subseteq E(M_{\U})$ such that $|S|,|S^\complement| \geq 2$. Then, 
$(S, S^\complement)$ is a 2-separation of $M_{\U}$ if and only if $(\phi_{\U}^{-1}(S), \phi_{\U}^{-1}(S^\complement))$ is a 2-separation of $M$.
\end{lemma}

\begin{proof}
Let $S$ be a subset of $E(M_{\U})$ such that $|S|,|S^\complement| \geq 2$. Suppose, first, that $(\phi_{\U}^{-1}(S), \phi_{\U}^{-1}(S^\complement))$ is a 2-separation of $M$. We show that $(S,S^\complement)$ is a 2-separation of $M_{\U}$. By Claim~\ref{localrestriction}, 
$$
\mbox{$M_{\U}|S = (M|\phi_{\U}^{-1}(S))_{\U_S} $ and $M_{\U}|S^\complement = (M|\phi_{\U}^{-1}(S^\complement))_{\U_{S^\complement}}$.}
$$ 
Let $B_{1,\U}$ be a base of $M_{\U}|S$ and $B_{2,\U}$ be a base of $M_{\U}|S^\complement$. As $B_{1,\U}$ is also a base of $(M|\phi_{\U}^{-1}(S))_{\U_S}$, there exists a corresponding base $B_1$ of $M|\phi_{\U}^{-1}(S)$ as in Lemma~\ref{localbases}. Similarly there exists a corresponding base $B_2$ of $M|\phi_{\U}^{-1}(S^\complement)$ for $B_{2,\U}$. Let $B$ be a base of $M$ such that $B\subseteq B_1\cup B_2$. As $(\phi_{\U}^{-1}(S),\phi_{\U}^{-1}(S^\complement))$ is a 2-separation of $M$, $|(B_1\cup B_2)\setminus B|=1$. Let $e$ be this element. Consider the corresponding base $B_{\U}$ of $M_{\U}$ for $B$ given by Lemma~\ref{localbases}. Surely, $B_{\U}\subseteq B_{1,\U} \cup B_{2,\U}$. Indeed, deleting $\phi_{\U}(e)$ from $B_{1,\U} \cup B_{2,\U}$ must give $B_{\U}$. Thus $(S,S^\complement)$ is a 2-separation.

Suppose, second, that $(S,S^\complement)$ is a 2-separation of $M_{\U}$. We show that $(\phi_{\U}^{-1}(S), \phi_{\U}^{-1}(S^\complement))$ is a 2-separation of $M$. Observe first that $|\phi_{\U}^{-1}(S)|$ and $|\phi_{\U}^{-1}(S^\complement)|$ are both at least~2. Let $B_1$ be a base of $M|\phi_{\U}^{-1}(S)$ and $B_2$ be a base of $M|\phi_{\U}^{-1}(S^\complement)$. Consider the corresponding bases $B_{1,\U}$ and $B_{2,\U}$ in $M_{\U}|S_{\U}$ and $M_{\U}|S_{\U}^\complement$. Let $B_{\U}$ be a base of $M_{\U}$ such that $B_{\U}\subseteq B_{1,\U}\cup B_{2,\U}$. As $S_{\U}$ is a 2-separation of $M_{\U}$, $|(B_{1,\U}\cup B_{2,\U})\setminus B_{\U}|=1$. Let $e$ be this element of $M_{\U}$. Consider the corresponding base $B$ of $M$ for $B_{\U}$ such that $B\subseteq B_1\cup B_2$. It follows that for all $f\in (B_1\cup B_2) \setminus B$, $\phi_{\U}(f)=e$. Without loss of generality, suppose that $e\in S_{\U}$. Let $X_i$ be the corresponding 2-separation from $\U$. Thus every such $f$ is in $\phi_{\U}^{-1}(S)$. So $f$ is in $B_1\setminus B$. But there is only one such $f$ because $B_1\cap X_i$ is a base of $M|X_i$, but $B\cap X_i$ is the base of a hyperplane of $M|X_i$. Hence, $|(B_1\cup B_2)\setminus B|=1$ and $(S, S^\complement)$ is a 2-separation of $M$. For all $i$, if $e_i\in S_{\U}$, then $X_i \subseteq S$. Similarly if $e_i\not\in S_{\U}$, then $X_i \subseteq S^\complement$. Hence, $(S, S^\complement)$ is a 2-separation of $M_{\U}$.
\end{proof}

By Lemma~\ref{local2seps}, if $(S, S^\complement)$ is a 2-separation of $M$, then $(\phi_{\U}(S),\phi_U(S)^\complement)$ is a 2-separation of $M_{\U}$ provided the latter two sets both have size at least~2. In fact, a stronger property holds; the next lemma asserts that ``reasonable" pairs of the form $(Y,Y^\complement)$ with $Y \subseteq \phi_{\U}(S)$ also form a 2-separation of $M_{\U}$. We make this precise next.    

\begin{lemma}\label{local2seps2}
Let $(S, S^\complement)$ be a 2-separation of a connected matroid $M$, and let $\phi_{\U}(S^\complement)^\complement \subseteq Y \subseteq \phi_{\U}(S)$ satisfy $|Y|, |Y^\complement| \geq 2$.  Then $(Y, Y^\complement)$ is a 2-separation of $M_{\U}$.
\end{lemma}

\noindent
{\bf Comment}: Note that $\phi_{\U}(S^\complement)$ and $\phi_{\U}(S)$ can only intersect in virtual elements. Then the set $\phi_{\U}(S^\complement)^\complement$ is simply $\phi_{\U}(S) \sm (\phi_{\U}(S) \cap \phi_{\U}(S^\complement))$. In particular, $\phi_{\U}(S) \sm Y$ consists solely of virtual elements.      

\begin{proof}
It is sufficient to prove that 
\begin{equation}\label{eq:1}
\mbox{there exists a 2-separation $(S', S'^\complement)$ of $M$ satisfying $\phi_{\U}(S')=Y$;} 
\end{equation}
observe that for such an $S'$ we have that $Y^\complement \subseteq \phi_{\U}(S'^\complement)$.
Assuming \eqref{eq:1}, the lemma follows via Lemma~\ref{local2seps}.

To prove \eqref{eq:1}, fix a well-ordering $(e_\alpha)_{\alpha <\gamma}$ on the elements of $\phi_{\U}(S) \setminus Y$ (all of which are virtual).  We prove the following claim.

\setcounter{thesubclaim}{0}
\begin{subclaim} \label{subcl:2}
There exists a sequence $(S_\alpha)_{\alpha < \gamma}$ of subsets of $E(M)$ such that for every $\alpha$ 
\begin{enumerate}
\item [\rm{(i)}] the pair $(S_\alpha, (S_\alpha)^\complement)$ is a 2-separation of $M$; 
\item [{\rm (ii)}] $\phi_{\U}(S^\complement)^\complement \subseteq Y \subseteq \phi_{\U}(S_\alpha)$; and 
\item [{\rm (iii)}] $S_\alpha\subseteq S_\beta$ and  $e_\beta \not\in \phi_{\U}(S_\alpha)$ for all $\beta < \alpha$.
\end{enumerate}
\end{subclaim}

\begin{innerproofof}{Subclaim 1}
We use transfinite induction. Let $S_1 = S$. The claim holds for $\alpha=1$. Let $\alpha >1$ and assume that the set $S_\beta$ has been defined for every $\beta < \alpha$.

If $\alpha$ is a successor ordinal, set 
$$
S_\alpha = S_{\beta} \cap (\phi_{\U}^{-1}(e_{\beta}))^\complement,
$$ 
where $\alpha$ is the successor of $\beta$.  
As $e_{\beta}\in \phi_{\U}(S) \setminus Y$, the set $\phi_{\U}^{-1}(e_{\beta})$ intersects both $S$ and $S^\complement$. Hence, $\phi_{\U}^{-1}(e_{\beta})$ has at least two elements and thus $e_{\beta}$ must be virtual. So
 $$(\phi_{\U}^{-1}(e_{\beta}), (\phi_{\U}^{-1}(e_{\beta}))^\complement)$$
 is a 2-separation of $M$. By the corner lemma (Lemma~\ref{corner}), $(S_\alpha, (S_\alpha)^\complement)$ is a 2-separation of $M$.  Note that $\phi_{\U}(S_\alpha) = \phi_{\U}(S_{\beta}) - e_{\beta}$ and that in this case the claim follows.

Suppose next that $\alpha$ is a limit ordinal, and set $S_\alpha = \bigcap_{\beta < \alpha} S_\beta$. By the infinite nested intersection lemma (see Lemma~\ref{nested}), $(S_\alpha, (S_\alpha)^\complement)$ is a 2-separation of $M$. 
Clearly, $\phi_{\U}(S^\complement)^\complement \subseteq Y$. Moreover, $Y \subseteq \phi_{\U}(S_\alpha)$ because $Y \subseteq \phi_{\U}(S_\beta)$ for every $\beta < \alpha$ by induction. Note that $S_\alpha \subseteq S_\beta$ for all $\beta < \alpha$. For all $\beta < \alpha$, as $e_{\beta} \not\in S_{\beta+1}$, then $e_\beta \not\in S_\alpha$ and the claim follows in this case as well.
\end{innerproofof}

Now $(S_\gamma, (S_\gamma)^\complement)$ is a 2-separation of $M$ such that $S_{\U}\subseteq \phi_{\U}(S_\gamma)$ and $e \not\in \phi_{\U}(S_\gamma)$ for all $e\in \phi_{\U}(S) \setminus Y$. Hence $Y = \phi_{\U}(S_\gamma)$ as desired and \eqref{eq:1} follows. 
\end{proof}

Finally, we are ready to prove Lemma~\ref{localgoodseps}.

\sloppy
\begin{proofof}{Lemma~\ref{localgoodseps}}
If $(S,S^\complement)$ is not a good 2-separation of $M_{\U}$, then there exists a 2-separation $(S',S'^\complement)$ of $M_{\U}$ crossing $(S, S^\complement)$. By Lemma~\ref{local2seps}, $(\phi_{\U}^{-1}(S'), \phi_{\U}^{-1}(S'^\complement))$ is a 2-separation of $M$; this 2-separation crosses $(\phi_{\U}^{-1}(S), \phi_{\U}^{-1}(S^\complement))$. Thus the latter separation is not good.

\fussy
If $(\phi_{\U}^{-1}(S), \phi_{\U}^{-1}(S^\complement))$ is not a good 2-separation of $M$, then there exists a 2-separation $(S',S'^\complement)$ of $M$ crossing it. Let $Y$ be a subset of $\phi_{\U}(S')$ such that $Y^\complement$ is a subset of $\phi_{\U}(S'^\complement)$ and $(Y, Y^\complement)$ crosses $(S,S^\complement)$.  By Lemma~\ref{local2seps2}, $(Y, Y^\complement)$ is a proper 2-separation of $M$. Hence, $(S,S^\complement)$ is not good.
\end{proofof}

%%%%%%%%%%%%%%%%%%%%%%%%%%%%%%%%%%%%%%%%%%%%%%%%%%%%%%%%%%%%

\section{Nested sequences of good 2-separations}\label{sec:inf-nested-main}
In this section we prove Proposition~\ref{end} asserting that {\sl the intersection of an infinite sequence of nested good 2-separations is always empty}. We do not know whether the lemma extends to $k$-separations for $k>2$.

\begin{proofof}{Proposition~\ref{end}} Assume towards a contradiction that $S_{\cap} = \bigcap_{i=1}^{\infty} S_i \not= \emptyset$. 
We may assume that $|S_{\cap}| = \{e\}$. Indeed, if $|S_{\cap}| \geq 2$, then  $(S_{\cap}, (S_{\cap})^\complement)$ is a 2-separation by the infinite nested intersection lemma (see Lemma~\ref{nested}); so $M = M_1 \oplus_2 M_2$, by Lemma~\ref{2-sum}, where $M_1$ has $E(M) \sm (\bigcap_{i=1}^{\infty} S_i)$ plus an additional element $e$ introduced by the 2-sum operation (see Lemma~\ref{2-sum}) as its ground set. $M_1$ then contains an infinite nested sequence $\{S_1' \supsetneq S_2' \supsetneq S_3' \ldots\}$ where $S_i' = (S_i + e) \setminus (E(M_2)-e)$ and $(S_i', (S_i')^\complement)$ is a good 2-separation of $M_1$ for every $i'$. Clearly, $\bigcap_{i'=1}^{\infty} S_{i'} = \{e\}$.  Hence, we may put $M=M_1$ and proceed assuming that $S_{\cap} = \{e\}$. 

In what follows, we show that $M$ is not a matroid. We establish this by constructing a base $B$ of $M-e$ such that $B+e$ is independent $M$. This then shows that $M$ is not connected; a contradiction.  

Choose a circuit $C$ containing $e$. As $M$ is connected, such a circuit exists and satisfies $C\ne \{e\}$. 
In fact,
\begin{equation}\label{specialcircuit}
\mbox{There does not exist an $i$ such that $C-e\subseteq S_i^\complement$.}
\end{equation}

\begin{innerproofof}{\eqref{specialcircuit}}
Assume towards a contradiction that $C-e\subseteq S_i^\complement$ for some $i$. 
Let $B_1$ be a base of $M|S_i^\complement$ containing $C-e$; such is a base of $M|S_i^\complement$ (as clearly $B_1$ spans $e$). Next, let $B_2$ be a base of $M|(S_i-e)$ and let $B_2'$ be a base of $M|S_i$ containing $B_2$; clearly $B_2' \subseteq B_2+e$. 

As $(S_i, S_i^\complement)$ is a 2-separation of $M$, there exists an element $f\in B_1\cup B_2'$ such that $B_1\cup B_2'-f$ is a base of $M$. If $B_2' = B_2+e$, then $f \in C$ and we may assume that $f =e$. 
In this case, $(S_i^\complement +e, S_i -e)$ is a $1$-separation of $M$ a contradiction to $M$ being connected. 
Suppose then that $B_2'=B_2$. Then, $(S_i^\complement +e, S_i -e)$ is a 2-separation of $M$ crossing $(S_{i+1}, (S_{i+1})^\complement)$, a contradiction.
\end{innerproofof}

Let $L_1 = S^\complement_1$ and put  $L_i = S^\complement_i - S^\complement_{i-1}$ for $i\geq 2$. We refer to $L_i$ as the $i$th \emph{block} of $M$. 
A corollary of \eqref{specialcircuit} is that there exist infinitely many $i$s for which $C\cap L_i \ne \emptyset$. Without loss of generality (by possibly discarding some of the $S_i$'s and redefining the blocks), we may assume that
\begin{equation}\label{specialcircuit2}
\mbox{$C\cap L_i \ne \emptyset$ for all $i$.}
\end{equation}

Let $U_1 = \{S_1\}$ and put $U_i = \{S_i, (S_{i-1})^\complement\}$ for $l \geq 2$. Let  $M_i$ denote the localization $M_{U_i}$; such satisfies $R(M_i) = L_i$ for every $i$ and is called \emph{real} if $R(M_i)$ spans $M_i$. 
We write $e_1$ to denote the virtual element of $M_1$ and $\{e_{i-1},e_i\}$ for those of $M_i$ when $i\geq 2$.  
Next, for $i \geq 1$ call $M_i \oplus_2 M_{i+1}$ \emph{real} whenever such is spanned by $E(M_i \oplus_2 M_{i+1}) \setminus \{e_{i-1},e_{i+1}\}$; naturally if $i=1$ we take $E(M_1 \oplus_2 M_2) -e_2$ instead.  We show that 
\begin{equation}\label{localblock1}
\mbox{$M_i \oplus_2 M_{i+1}$ is real $\iff$  $M_i$ or $M_{i+1}$ is real.}
\end{equation}

\begin{innerproofof}{\eqref{localblock1}}
Suppose $M_i$ is real. Let $B_i$ be a base of $M_i$ contained in $L_i$, and  let $B_{i+1}$ be a base of $M_{i+1}\setminus e_{i+1}$ containing $e_{i}$. As $e_i, e_{i+1}\in \phi_{U_{i+1}}(C)$, there is a circuit in $M_{i+1}$ containing $e_i$ and $e_{i+1}$. Consequently, $B_{i+1}$ is a base of $M_{i+1}$. Put $B = B_i \cup (B_{i+1} - e_i)$ and note that this is a base of $M_i \oplus_2 M_{i+1}$ such that $e_{i-1}, e_{i+1} \not\in B$. Hence, $M_i \oplus_2 M_{i+1}$ is real as desired. An analogous argument holds when $M_{i+1}$ is real.

Suppose then that $M_i \oplus_2 M_{i+1}$ is real. Let $B$ be a base of $M_i \oplus_2 M_{i+1}$ containing neither of $e_{i-1}$ and $e_{i+1}$. As $(S_i, S_i^\complement)$ is a 2-separation of $M$, $(E(M_i)-e_i, E(M_{i+1})-e_i)$ is a 2-separation of $M_i \oplus_2 M_{i+1}$, by Lemma~\ref{local2seps}. Thus, either $B_1=B\cap E(M_i)-e_i$ is a base of $M_i-e_i$ or $B_2=B\cap E(M_{i+1})-e_i$ is a base of $M_{i+1}-e_i$. In the former case, $B_1$ is also a base of $M_i$ and yet $e_{i-1}, e_i \not\in B_1$ as desired. Similarly in the latter case, $B_2$ is a base of $M_{i+1}$ and yet $e_i, e_{i+1}\not\in B_2$ as desired. 
\end{innerproofof}

We may now prove that 
\begin{equation}\label{localblock2}
\mbox{$M_i$ or $M_{i+1}$ is real, for any $i$.}
\end{equation}

\begin{innerproofof}{\eqref{localblock2}} 
Assume towards a contradiction that $M_i$ and $M_{i+1}$ are not real. 
Then, $M' = M_i \oplus_2 M_{i+1}$ is not real, by \eqref{localblock1}. As every base of $M'$ includes one of $e_{i-1}, e_{i+1}$ and as $L_i$ and $L_{i+1}$ are nonempty, by \eqref{specialcircuit2}, it follows that $(\{ e_{i-1}, e_{i+1}\}, L_i\cup L_{i+1})$ is a 2-separation of $M'$. 

By Lemma~\ref{local2seps}, $( E(M)- (L_i\cup L_{i+1}), L_i\cup L_{i+1})$ is a 2-separation of $M$. As such crosses $(S_i,S_i^\complement)$ we attain a contradiction to the assumption that $(S_i, S_i^\complement)$ is good. 
\end{innerproofof}

A corollary of \eqref{localblock2} is that there exist infinitely many $i$ such that $M_i$ is real. Without loss of generality (by discarding some of the $S_i$'s and redefining blocks), we may assume, by \eqref{localblock2}, that
\begin{equation}\label{localblock3}
\mbox{$M_i$ is real, for all $i$.}
\end{equation}

Let $B_1\in \B(M_1)$ such that $e_1\in B_1$. For all $k\ge 1$, let $B_{2k}\in \B(M_{2k})$ such that $e_{2k-1},e_{2k}\not\in B_{2k}$; such exists as $M_{2k}$ is real. For all $k \ge 1$, let $B_{2k+1}\in \B(M_{2k+1})$ containing $\{e_{2k}, e_{2k+1}\} $. Such a base exists as $\{e_{2k}, e_{2k+1}\}$ is independent in $M_{2k}$. The latter follows from the fact that  $\{e_{2k}, e_{2k+1}\}$ is a proper subset of $\phi_{U_{2k+1}}(C)$ as $C\cap L_{2k+1} \ne \emptyset$. Define $B = \bigcup_i B_i\cap L_i$, and observe that

\begin{equation}\label{localblock4}
\mbox{$B+e$ is independent.}
\end{equation}

\begin{innerproofof}{\eqref{localblock4}}
Suppose not. Then there exists a circuit $C$ of $M$ contained in $B+e$. Let $i$ be minimum such that $C\cap L_i \ne \emptyset$. Note that, as $B\cap L_i= B_i\cap L_i$ is independent, either there exists $j\ne i$ such that $C\cap L_j\ne \emptyset$, or $e\in C$. If $i$ is odd, then $C'=\phi_{U_i}(C)$ has at least two elements, one in $L_i$ and $e_i$. Hence, $C'$ is a circuit of $M_i$ contained in $L_i+e_i \subseteq B_i$, a contradiction. If $i$ is even, then $C'=\phi_{U_{i+1}}(C)$ has at least two elements, $e_i$ and either one in $L_i$ or $e_{i+1}$. Hence, $C'$ is a circuit of $M_{i+1}$ contained in $L_i+e_i+e_{i+1} \subseteq B_{i+1}$, a contradiction.
\end{innerproofof}

To conclude we show that 
\begin{equation}\label{localblock5}
\mbox{$B$ spans $M-e$.}
\end{equation}

\begin{innerproofof}{\eqref{localblock5}}
Let $v\in M-e$. Hence $v\in L_i$ for some $i$. To show that $B$ spans $v$, we prove that $B+v$ has a circuit $C$ containing $v$.  Consider the fundamental circuit $C(v,B_i)$ in $M_i$. If $C(v,B_i)\subseteq L_i$, then $C(v,B_i)\subseteq B_i \cap L_i + v \subseteq B+v$ as desired. Assume then that $C(v,B_i) \setminus L_i \ne \emptyset$. Thus, $i$ is odd as otherwise $B_i \subseteq L_i$. If $i=1$, then $C(v,B_1) \setminus L_1 \subseteq \{e_1\}$ and hence $e_1\in C(v,B_1)$. But then $C(v,B_1)\cup C(e_1,B_2) - e_1$ is a circuit of $M$ contained in $B_1\cup B_2 + v - e_1 \subseteq B+v$ as desired.
 
Suppose then that $i\ne 1$. Then, $C(v,B_i) \setminus L_i \subseteq \{e_{i-1}, e_i\}$. If $C(v,B_i) \setminus L_i=\{e_{i-1}\}$, then $C(e_{i-1}, B_{i-1}) \cup C(v,B_i) - e_{i-1}$ is a circuit of $M$ as desired. Similarly if $C(v,B_i) \setminus L_i=\{e_i\}$, , then $C(v, B_i) \cup C(e_i,B_{i+1}) - e_i$ is a circuit of $M$ as desired. Finally, if $C(v,B_i) \setminus L_i=\{e_{i-1},e_i\}$, then $C(e_{i-1}, B_{i-1}) \cup C(v,B_i) \cup C(e_i,B_{i+1}) - e_{i-1} - e_i$ is a circuit of $M$ as desired. 
\end{innerproofof}

By \eqref{localblock4}, $B$ is a base of $M-e$. However, by \eqref{localblock5}, $B+e$ is independent. Thus, $M$ is not connected, a contradiction. This proves Proposition~\ref{end}.
\end{proofof}

%%%%%%%%%%%%%%%%%%%%%%%%%%%%%%%%%%%%%%%%%%%%%%%%%%%%%%%%%%%%%%%%%%%%%%%%%%%%%%%%%%%%%%%%%%%%%%%%%%%%%%%%%%

\section{The structure of torsos}\label{sec:torso-structure}
In this section we prove Lemma~\ref{prim-main}; recall that this lemma asserts that {\sl a connected matroid with no good 2-separations is $3$-connected, a circuit, or a cocircuit}, and note that the converse of this lemma is trivial.  As mentioned in Section~\ref{sec:outline}, our proof of Lemma~\ref{prim-main} is carried out in two steps; these are captured by Lemmas~\ref{every} and~\ref{prim} that together imply Lemma~\ref{prim-main}. This general two step framework is that of Cunningham and Edmonds~\cite{CE}. The proof of Lemma~\ref{every} is simple. The proof of Lemma~\ref{prim}, however, requires effort and new ideas.

%\subsection{Fragile matroids}\label{sec:every} 
%In this section, we prove Lemma~\ref{every}. A connected matroid is called \emph{fragile} if it satisfies the property that for every two of its elements, say $x$ and $y$, the pair $(\{x,y\}, \{x,y\}^\complement)$ forms a 2-separation of the matroid.
%With this terminology, Lemma~\ref{every} reads as follows. 

%\begin{lemma}
% A connected matroid is fragile if and only if it is a circuit or a cocircuit.  
% \end{lemma}
 
Lemma~\ref{every} is a consequence of~\cite[Corollary 8.1.11]{OxleyBook}. Since in~\cite{OxleyBook} a proof of the latter is not provided, we include one here for completeness. We shall make use of the following fact. 

\begin{observation}\label{obs:small}
A $k$-connected matroid $M$ with $|E(M)| \geq 2(k-1)$ has all its circuits and cocircuits of size $\geq k$.
\end{observation}

\begin{proof}
Indeed, a circuit of size $j$ gives rise to a $j$-separation. 
\end{proof}

\begin{lemma}\label{every2} \emph{\cite[Corollary 8.1.11]{OxleyBook}}\\
If $(S, S^\complement)$ is a $k$-separation of a $k$-connected matroid $M$ with $|S|=k$, then $S$ is either a coindenpendent circuit or an independent cocircuit.
\end{lemma}

\begin{proof}
Suppose that $S$ is coindependent. Then, $S^\complement$ spans $M$ and contains a base $B$ of $M$.  Let $B_S$ be a base of $M|S$. As $M$ is $k$-connected and $(S,S^\complement)$ is a $k$-separation, we must remove exactly $k-1$ elements from $B\cup B_S$ in order to obtain a base of $M$. Thus $|B_S|=k-1$. As $|S|=k$, $S$ must be a circuit of size $k$, by Observation~\ref{obs:small}. Dually, if $S$ is independent then $S$ is a cocircuit.

So we may assume that $S$ is dependent and codependent. Since a $k$-connected matroid admitting a proper $k$-separation satisfies $|E(M)| \geq 2k$, it follows that $S$ is a circuit and a cocircuit, by Observation~\ref{obs:small}.  Let $v\in S$. As $S-v$ is coindependent, $S^\complement + v$ contains a base $B$ of $M$. As $S$ is not coindependent, $v\in B$. Let $u\ne v$ with $u\in S$. Note that $S-u$ is independent but $S$ is not. Hence, $S-u$ is a base of $S$. Meanwhile, $B-v$ is a base of $S^\complement$. However, $|(B-v)\cup (S-u) \setminus B| = k-2$. So $(S,S^\complement)$ is a $(k-1)$-separation, contradicting that $M$ is $k$-connected. \end{proof}

We are now ready to prove Lemma~\ref{every}.

\begin{proofof}{Lemma~\ref{every}} 
By assumption, $(\{x,y\}, \{x,y\}^\complement)$ is a 2-separation of $M$ for each pair $\{x,y\} \subseteq E(M)$.  By Lemma~\ref{every2}, every such pair is then either a circuit or a cocircuit. 
As a circuit $C$ and a cocircuit $C^*$ never satisfy $|C \cap C^*|=1$ \cite[Lemma $3.1$]{InfMatroidAxioms}, either every pair is a circuit or every pair is a cocircuit; hence $M$ is either a cocircuit or a circuit, respectively.  
\end{proofof}

\subsection{Non-3-connected primitive matroids}\label{sec:non-3-primitive}
In this section, we prove Lemma~\ref{prim}. Let us call a connected matroid \emph{primitive} if it has no good 2-separations.  With this terminology, Lemma~\ref{prim} reads as follows. 

\begin{lemma}
If $M$ is a primitive matroid that is not 3-connected, then for every two elements $x,y$ the partition $(\{x,y\}, \{x,y\}^\complement)$ is a 2-separation.
\end{lemma}

For the remainder of this section, let {\sl $M$ denote a primitive matroid that is not $3$-connected}. 
The goal of this section then is to show that $(\{x,y\},\{x,y\}^\complement)$ is a 2-separation of $M$ for every pair of its elements $x$ and $y$. The first step in our proof is Lemma~\ref{primpair} stated below;  this lemma asserts that for any pair of elements $x$ and $y$, the matroid $M$ admits a 2-separation with $x$ and $y$ on opposite sides of the separation. We shall see that Lemma~\ref{primpair} is a consequence of the following lemma.     

\begin{lemma}\label{primseparating}
Let $X \subseteq E(M)$, $|X|\ge 2$ and $(S,S^C)$ be a 2-separation such that $X\subseteq S$.
Then, there exist two crossing 2-separations $(S',S'^\complement)$ and $(U, U^\complement)$ satisfying $X\subseteq S'\subseteq S$ and $X\cap U, X\cap U^\complement \ne \emptyset$.
\end{lemma}
 
\begin{proof}  
Suppose that the claimed pair of crossing 2-separations does not exist. We shall define for every ordinal $\alpha$ a set $S_\alpha\subseteq E(M)$ such that $S_\alpha \subsetneq S_{\alpha'}$ whenever $\alpha > \alpha'$. This will yield a contradiction as soon as $\alpha$ is bigger than the cardinality of $E(M)$.

Set $S_0 = S$. Now, consider $\alpha > 0$ and assume that $S_{\alpha'}$ has been defined for every $\alpha' < \alpha$. Whenever possible, let  
\begin{equation}\label{prim.2}
S_{\alpha} = \left\{ \begin{array} {ll}
                                  S_{\beta}\cap U_{\alpha} & \mbox{whenever $\alpha$ is the successor of ordinal $\beta$}\\
                                  \bigcap_{\beta<\alpha} S_{\beta} & \mbox{whenever $\alpha$ is a limit ordinal,}
                              \end{array}
                     \right.
\end{equation}
where $(U_{\alpha},(U_{\alpha})^\complement)$ is some 2-separation crossing $(S_{\beta},(S_{\beta})^\complement)$. If such a set~$U_{\alpha}$ exists, it will contain~$X$, as otherwise $S'=S_\beta$ and $U=U_\alpha$ would be as desired in the lemma.%
   \COMMENT{}
   Hence $X\subseteq S_\alpha$ whenever $S_\alpha$ is defined.

To conclude the proof we show that $S_\alpha$ is defined for every~$\alpha$. This will be the case as soon as $U_\alpha$ exists at every successor step $\alpha = \beta+1$. Which it does as soon as $(S_{\beta}, (S_{\beta})^\complement)$ is a 2-separation: then our assumption that $M$ has no good 2-separations implies that there is a 2-separation crossing $(S_{\beta}, (S_{\beta})^\complement)$, and we can take this as  $(U_{\alpha},(U_{\alpha})^\complement)$. It thus remains to prove that  

\begin{equation}\label{alphaproper}
\mbox{for all $\alpha$, $(S_\alpha, (S_\alpha)^\complement)$ is a 2-separation of $M$.}
\end{equation} 

To prove this we use transfinite induction. If $\alpha = \beta+1$ is a successor ordinal, then $S_\alpha = S_{\beta}\cap U_\alpha$. By induction $(S_{\beta}, (S_{\beta})^\complement)$ is a 2-separation. By definition, so is $(U_{\alpha},(U_{\alpha})^\complement)$. Thus, by the corner lemma (Lemma~\ref{corner}), $(S_\alpha, (S_\alpha)^\complement)$ is a 2-separation of $M$. 

If $\alpha$ is a limit ordinal, then $S_\alpha= \bigcap_{\beta<\alpha} S_{\beta}$. By induction, $(S_{\beta}, (S_{\beta})^\complement)$ is a 2-separation for all $\alpha$. By the infinite nested intersection lemma (see Lemma~\ref{nested}), $(S_{\alpha}, (S_{\alpha})^\complement)$ is a 2-separation of $M$ as $|S_\alpha| \geq |X| \geq 2$ and $|(S_\alpha)^\complement| \ge 2$; \eqref{alphaproper} now follows and consequently the lemma.
\end{proof}

Lemma~\ref{primpair} is a consequence of Lemma~\ref{primseparating}.

\begin{lemma}\label{primpair}
For all distinct $u,v \in E(M)$ there is a 2-separation $(S,S^\complement)$ satisfying $u \in S$ and $v \in S^\complement$.
\end{lemma}

\begin{proof}
Let $X=\{u,v\}$. As $M$ is not $3$-connected, there exists a 2-separation $(S,S^\complement)$ of $M$. If $X$ intersects both $S$ and $S^\complement$, Lemma~\ref{primpair} follows. So we may assume without loss of generality that $X\subseteq S$. Applying Lemma~\ref{primseparating}, there exists a 2-separation $(U,U^\complement)$ such that $X$ intersects both $U$ and $U^\complement$ and Lemma~\ref{primpair} follows. 
\end{proof}

We shall require the following property that is stronger than Lemma~\ref{primpair}.

\begin{lemma}\label{primtriple} 
For all distinct $x,y,z \in E(M)$ there is a 2-separation $(S, S^\complement)$ of $M$ satisfying $\{x,y\} \subseteq S$ and $z \in S^\complement$.
\end{lemma} 

\begin{proof}
By Lemma~\ref{primpair}, there exists a 2-separation $(S,S^\complement)$ satisfying $x \in S^\complement$ and $z \in S$. We may assume that $y \in S$ as otherwise Lemma~\ref{primtriple} follows. Let $X=\{y,z\}$.  Applying Lemma~\ref{primseparating} to $X$ and $(S,S^\complement)$, there exists two crossing 2-separations $(S', S'^\complement)$ and $(U,U^\complement)$ such that $X\subseteq S'\subseteq S$ and $U\cap X, U\cap X^\complement\ne \emptyset$. We may assume without loss of generality that $y\in U$ and $z\in U^\complement$. Note that $x\in S'^\complement$. If $x$ is in $U$, then $(U,U^\complement)$ is the desired separation. So suppose that $x\in U^\complement$. By the symmetric difference lemma (Lemma~\ref{sim}), $(S \Delta U^\complement, (S \Delta U^\complement)^\complement)$ is a 2-separation of $M$ with $\{x,y\} \subseteq S \Delta U^\complement$ and $z\in (S\Delta U^\complement)^\complement$ as desired.
\end{proof}

%\subsubsection{Primitive matroids are $3$-connected, circuits, or cocircuits} 
We now prove Lemma~\ref{prim}. 

\begin{proofof}{Lemma~\ref{prim}}
Let $M$ be a primitive matroid and assume towards a contradiction that there exists a pair of elements, say $x$ and $y$, such that $(\{x,y\}, \{x,y\}^\complement)$ is not a 2-separation of $M$. Choose $u\in \{x,y\}^\complement$. By Lemma~\ref{primtriple}, there exists a 2-separation $(S_1,S_1^\complement)$ of $M$ such that $x,y \in S_1$ and $u \not\in S_1$. 

\setcounter{thesubclaim}{0}
\begin{subclaim}\label{existence}
There exists a sequence of sets $(S_\alpha)_{\alpha}$ such that 
\begin{enumerate}
\item [{\rm (i)}] $x,y \in S_\alpha$ for every $\alpha$; 
\item [{\rm (ii)}] $S_\beta \supsetneq S_\alpha$ for every $\beta < \alpha$ (unless $S_\beta=S_\alpha=\{x,y\}$); and 
\item [{\rm (iii)}] $(S_\alpha, (S_\alpha)^\complement)$ is a 2-separation of $M$ for every $\alpha$.
\end{enumerate}
\end{subclaim}

\begin{innerproofof}{Subclaim 1}
We prove the claim via transfinite induction. Note that the claim holds for $\alpha=1$ and assume that it holds for all $\beta < \alpha$. If $\alpha = \beta+1$ is a successor ordinal, then set $\U_\alpha = \{(S_{\beta})^\complement \}$ and consider the localization of $M_{\U_\alpha}$ of $M$ at $\U_{\alpha}$. If $M_{\U_\alpha}$ has a good 2-separation then so does $M$, by Lemma~\ref{localgoodseps}, which is a contradiction. 
Now, by definition of $M$, there exists a 2-separation $(S', S'^\complement)$ of $M$ crossing $(S_{\beta}, (S_{\beta})^\complement)$. If $x,y\in S'$, then set $S_\alpha=S'\cap S_{\beta}$. By the corner lemma (Lemma~\ref{corner}), $(S_\alpha, S_\alpha^\complement)$ is a 2-separation of $M$. Moreover, $x,y \in S_\alpha$ and $S_{\beta} \supsetneq S_{\alpha}$ and the statement follows. So we may assume by symmetry and without loss of generality that $x\in S'$ and $y\in S'^\complement$.

As $(\{x,y\}, \{x,y\}^{\complement})$ is not a 2-separation of $M$, by assumption, there exists a $v\in S_{\beta}-\{x,y\}$. We may assume, without loss of generality, that $v\in S'$. Let $w=\phi_{\U}(S_{\beta}^\complement)$. By Lemma~\ref{local2seps2}, there exists a 2-separation $(S_{\U_{\alpha}}, (S_{\U_{\alpha}})^\complement)$ of $M_{\U_{\alpha}}$ such that $x,v \in S_{\U_\alpha}$ and $y,w \in (S_{\U_\alpha})^\complement$. Hence, $M_{\U_\alpha}$ is not $3$-connected. By Lemma~\ref{primtriple} applied to $M_{\U_\alpha}$, there exists a 2-separation $((S_\alpha)_\U, (S_\alpha)_{\U}^\complement)$ such that $x,y \in (S_\alpha)_\U$ and $w\in (S_\alpha)_{\U}^\complement$. Let $S_\alpha=\phi_{\U_\alpha}^{-1}( (S_\alpha)_{\U} )$. By Lemma~\ref{local2seps}, $(S_\alpha, (S_\alpha)^\complement)$ is a 2-separation of $M$. Clearly, $x,y \in S_\alpha$. Moreover, $S_{\beta} \supsetneq S_{\alpha}$ because $|(S_\alpha)_{\U}^\complement|\ge 2$.

Next, consider the case that $\alpha$ is a limit ordinal and define $S_\alpha = \bigcap_{\beta < \alpha} S_\beta$. By Lemma \ref{nested}, $S_\alpha$ is a 2-separation of $M$. By induction, $x,y \in S_\beta$ for all $\beta < \alpha$. Hence, $x,y\in S_\alpha$. Furthermore, by induction $S_{\beta+1} \supsetneq S_{\beta}$ for all $\beta <\alpha$. By definition, $S_\alpha \subseteq S_{\beta+1}$. Hence, $S_\beta \supsetneq S_\alpha$ for all $\beta < \alpha$ and the statement follows.
\end{innerproofof}

By Subclaim~\ref{existence} there exists an ordinal $\alpha$ for which $S_\alpha = \{x,y\}$. Lemma~\ref{prim} follows.
\end{proofof}

%%%%%%%%%%%%%%%%%%%%%%%%%%%%%%%%%%%%%%%%%%%%%%%%%%%%%

\section{Constructing a decomposition tree}\label{sec:decomposition-main}
The goal of this section is to prove Lemma~\ref{lem:td-partial} stated below. Prior to stating this lemma, let us first be reminded of some of the notation and terminology set in Section~\ref{sec:outline}. Let $\G = \G(M)$ be the set comprised of all the good 2-separation of a connected matroid $M$; this is a set of nested 2-separations of $M$ with the property that $(A,A^\complement) \in \G$ implies $(A^\complement,A) \in \G$. 
Define a partial ordering on $\G$ given by writing $(A,A^\complement) \le (B,B^\complement)$ whenever $A\subseteq B$. Next, call $(A,A^\complement)$ and $(B,B^\complement)$ {\sl equivalent}, and write $(A,A^\complement) \sim (B,B^\complement)$, if either $(A,A^\complement) = (B,B^\complement)$ or $(A^\complement,A)$ is a predecessor of $(B,B^\complement)$ in this ordering.
Finally, let $T_{\G}$ and $R_{\G}$ be as defined in \eqref{eq:ET}, \eqref{eq:VT}, and \eqref{eq:Rv}. 

This section is dedicated to the proof of the following lemma.

%\begin{lemma}\label{lem:td-partial}
%If $\F$ is the set of all good 2-separations of a connected matroid $M$, then $(T_{\F},R)$ is an irredundant tree-decomposition of uniform adhesion 2, and with all of its torsos primitive
%\end{lemma}

\begin{lemma}\label{lem:td-partial}
$(T_{\G},R_{\G})$ is an irredundant tree-decomposition of uniform adhesion 2 with all of its torsos primitive.
\end{lemma}

\noindent
In what follows, we verify each of the properties claimed for $(T_{\G},R_{\G})$ in Lemma~\ref{lem:td-partial}.
We begin by establishing that the vertices of $T_{\G}$ are properly defined. That is, we prove that 
\begin{equation}\label{eq:sim}
\mbox{the relation $\sim$ is an equivalence relation on $\G$}.
\end{equation}

\begin{proofof}{\eqref{eq:sim}}
By definition, the relation $\sim$ is reflexive and symmetric. We prove that $\sim$ is transitive. Suppose then that $(A,A^\complement)\sim (B, B^\complement)$ and that $(B, B^\complement) \sim (C, C^\complement)\in \G$. Assume now towards contradiction that $(A, A^\complement) \not\sim (C, C^\complement)$. Then, $A\ne C$ and there exists $(D, D^\complement)\in \G$ such that $(A^\complement,A) > (D,D^\complement) > (C^\complement,C)$. That is, $A\supsetneq D \supsetneq C^\complement$. Clearly, $B\ne A$ and $B\ne C$. By definition of the relation $\sim$, we have that $A\supsetneq B^\complement$ and $C\supsetneq B^\complement$. Now, $D$ does not contain $B$ as $A^\complement \subsetneq B$. Similarly $D^\complement$ does not contain $B$ as $C^\complement \subsetneq B$. As $\G$ is nested, either $D$ or $D^\complement$ must contain $B^\complement$. Without loss of generality, suppose that $D$ contains $B^\complement$. As $C^\complement$ is a subset of $D$ and $C$ contains $B^\complement$, it follows that $D \supsetneq B^\complement$. But then $(A,A^\complement) > (D,D^\complement) > (B^\complement, B)$, a contradiction.
\end{proofof}

Next, we prove that $T_{\G}$ is a tree. 

\begin{claim} \label{tree}
$T_{\G}$ is a tree.
\end{claim}

\begin{proof}
To prove that $T_{\G}$ is a tree, we show that $T_{\G}$ is connected and acyclic. Suppose that $T_{\G}$ had a cycle $v_1v_2\ldots v_n$. Let $\{A_i,A_i^\complement\}$ represent the edge between $v_i$ and $v_{i+1}$ where values are taken modulo $n$. We may then assume without loss of generality that $A_i \subseteq A_{i+1}$ where values are taken modulo $n$. But then certainly, all these sets are equal, in which case all of the 2-separations $(A_i, A_i^\complement)$ are incident, but this is impossible if $n\ge 3$, a contradiction. Thus $T_{\G}$ is acyclic.

To prove $T_{\G}$ is connected, we show that the unique path between any two distinct nodes $u, v\in V(T_{\G})$ is finite. Let $(A,A^\complement)$ be a member of the incidence class corresponding to $v$ such that $R_v \subseteq A$ and $R_u \subseteq A^\complement$. Similarly let $(B,B^\complement)$ be a member of the incidence class corresponding to $v$ such that $R_v \subseteq B$ and $R_u \subseteq B^\complement$. Consider a maximal sequence $(A,A^\complement) > (S_1,S_1^\complement) > (S_2,S_2^\complement) > \ldots > (B,B^\complement)$ where, for all $i$, $(S_i,S_i^\complement)\in \G$ is a good 2-separation of $M$. By Proposition~\ref{end}, this sequence is finite. This sequence corresponds to a path in $T_{\G}$.
\end{proof}

Next, we consider the adhesion of $(T_{\G},R_{\G})$ and prove the following.

\begin{claim}\label{adhesion}
$(T_{\G},R_{\G})$ has uniform adhesion~2.
\end{claim}

\begin{proof}
The following two claims facilitate our proof. 

\setcounter{thesubclaim}{0}
\begin{subclaim}\label{partition}
The sets $(R_v)_{v\in V(T_{\G})}$ partition $E(M)$.
\end{subclaim}

\begin{innerproofof}{Subclaim 1}
Suppose we are given $x\in E(M)$. We want to find $v\in \V(T_{\G})$ with $x\in R_v$. To do this, orient every edge $\{A, A^\complement\}$ of $T_{\G}$ towards $[(A,A^\complement)]$ if $x\in A$ and towards $[(A^\complement,A)]$ if $x\in A^\complement$. Let $v$ be a sink under this orientation. A sink exists as otherwise there would be an infinite directed path $[(S_1,S_1^\complement)], [(S_2, S_2^\complement)], \ldots$ with $S_1 \supsetneq S_2 \supsetneq \ldots$ where for all $i$, $(S_i,S_i^\complement)\in \G$ is a good 2-separation and $x\in S_i$. Thus, $\bigcap_i S_i \supseteq \{x\}$, contradicting Proposition~\ref{end}. So $v$ exists, which implies that for all $(A,A^\complement) \in \G$ in the incidence class of $v$, $x\in A$. By definition then, $x\in R_v$.

Conversely, we show that $R_v \cap R_w = \emptyset$ for distinct nodes $v,w$. As $T_{\G}$ is connected by Claim~\ref{tree}, let $\{A,A^\complement\}$ be an edge along the path between $v$ and $w$ such that $R_v \subseteq A$ and $R_w \subseteq A^\complement$. So $R_v\cap R_w=\emptyset$.
\end{innerproofof}

Recall now that given an edge $e=vw$ of~$T_{\G}$, we write $T_v$ and $T_w$ for the components of $T_{\G}-e$ containing $v$ and~$w$, respectively, and that $S(e,v) = \bigcup_{u\in T_v} R_u$ and that $S(e,w) = \bigcup_{u\in T_w} R_u$, where $R_u$ is as in \eqref{eq:Rv}. 

\begin{subclaim}\label{edges}
For every edge $e=\{A,A^\complement\}$ of  $T_{\G}$, $S(e,[(A,A^\complement)])=A$ and $S(e,[(A^\complement,A)])=A^\complement$.
\end{subclaim}

\begin{innerproofof}{Subclaim 2}
First, we show that $S(e, [(A,A^\complement)]) \subseteq A$. Let $T_A$ denote the component of $T_{\G} - e$ containing the node $[(A,A^\complement)]$. It suffices to prove that $R_u \subseteq A$ for all $u\in T_A$. We prove this by induction on the length of the path $P$ from $u$ to $[(A,A^\complement)]$ in $T_A$. If $u=[(A,A^\complement)]$, then by definition $R_u \subseteq A$. So we may assume that $P$ has at least one edge. Let $f=\{B,B^\complement\}$ be the edge in $P$ incident with $[(A,A^\complement)]$. We may assume without loss of generality that $(B^\complement, B)$ is incident with $(A,A^\complement)$. That is, $B \subsetneq A$. Now $P-f$ is a path from from $u$ to $[(B,B^\complement)]$ with smaller length than $P$. By induction, $R_u \subseteq B$. Hence $R_u \subseteq A$ as desired. 

By symmetry $S(e,[(A^\complement, A)]) \subseteq A^\complement$. By Subclaim~\ref{partition}, $S(e, [(A,A^\complement)])$ and $S(e, [(A^\complement, A)])$ partition $E(M)$. Hence, it follows that $S(e,[(A,A^\complement)])=A$ and $S(e,[(A^\complement,A)])=A^\complement$.
\end{innerproofof}

Claim~\ref{adhesion} now follows. 
\end{proof}

Recall that for a vertex $v$ of $T_{\G}$, we write $M_v$ to denote the torso of $(T_{\G},R_{\G})$ associated with $v$. We prove the following. 

\begin{claim} \label{vertexprim}
For every vertex $v$  of $T_{\G}$, the torso $M_v$ has no good 2-separations. 
\end{claim}

\begin{proof}
Suppose that $M_v$ has a good 2-separation $(S,S^\complement)$, and set 
\begin{equation*}
\U=\{S(e,v) | \; \mbox{$e$ is an edge incident with $v$}\}.
\end{equation*}
Then, $M_v$ is equal to the localization of $M$ at $\U$. By Lemma~\ref{localgoodseps}, $(\phi_{\U}^{-1}(S), \phi_{\U}^{-1}(S^\complement))$ is a good 2-separation of $M$. As $|S|, |S^\complement| \ge 2$, this separation does not correspond to an edge of $T_{\G}$, a contradiction.
\end{proof}

The irredundancy of $(T_{\G}, R_{\G})$ is considered next. 

\begin{claim}
$(T_{\G}, R_{\G})$ is irredundant.
\end{claim}

\begin{proof}
Suppose not. Then there exists an edge $uv$ of $T_{\G}$ such that, without loss of generality, the torsos $M_u$ and $M_v$ are circuits. By the construction of $T_{\G}$, $S(e,u)$ is a good 2-separation in $M$. Put $\U = \{S(f,u): f\ne e\} \cup \{S(f,v): f\ne e\}$, and consider the localization $M_{\U}$ of $M$ at $\U$.
Set $S_{\U} = \phi_{\U}(S(e,u))$. By Lemma~\ref{localgoodseps}, $(S_{\U}, S_{\U}^\complement)$ is a good 2-separation of $M_{\U}$. However, $M_{\U}$ is simply the 2-sum of the torsos $M_u$ and $M_v$ along the element $e$. Note that the 2-sum of two circuits is a circuit and that the 2-sum of two cocircuits is a cocircuit, implying that $M_{\U}$ has no good 2-separations, which is a contradiction.
\end{proof}

Finally, we show that any other irredundant tree-decomposition of $M$ whose torsos are primitive is isomorphic to $(T_{\G},R_{\G})$ as specified in Theorem~\ref{main}(ii).

\begin{claim}
$(T_{\G},R_{\G})$ is the unique irredundant tree-decomposition of $M$ whose torsos are primitive.
\end{claim}

\begin{proof}
To show uniqueness of $(T_{\G},R_{\G})$, it suffices to prove that any irredundant decomposition tree $T'$ of $M$ is isomorphic to $T_{\G}$. We show that every edge of $T'$ is a good 2-separation in $M$ and that every good 2-separation in $M$ is an edge of $T'$. 

Suppose that there exists an edge $e=uv$ of $T'$ such that $(S(e,u), S(e,v))$ is not a good 2-separation of $M$. Let $\U = \{S(f,u): f\ne e\} \cup \{S(f,v): f\ne e\}$, and consider the localization $M_{\U}$ of $M$ at $\U$. By Lemma~\ref{localgoodseps}, $(\phi_{\U}(S(e,u)), \phi_{\U}(S(e,v)))$ is a not a good 2-separation of $M_{\U}$. However, $M_{\U}$ is simply the 2-sum of the torsos $M_u$ and $M_v$ along the element $e$. As $T'$ is a decomposition tree, the torsos $M_u$ and $M_v$ have no good 2-separations. As $T'$ is irrendundant, either one of them is $3$-connected, or one is a circuit and the other is a cocircuit, by Lemma~\ref{prim-main}. In either case, $(\phi_{\U}(S(e,u)), \phi_{\U}(S(e,v)))$ is a good 2-separation of $M_{\U}$, a contradiction.

Suppose that there exists a good 2-separation $(S,S^\complement)$ of $M$ that is not an edge of $T'$. As $(S, S^\complement)$ is nested with all the edges of $T'$. There must exist a vertex $v$ of $T_{\G}$ such that $|\phi_{\U}(S)|, |\phi_{\U}(S^\complement)| \ge 2$, $\phi_{\U}(S) \cap \phi_{\U}(S^\complement) = \emptyset$ where the localization $M_{\U}$ is equal to the torso $M_v$.  By Lemma~\ref{localgoodseps}, $(\phi_{\U}(S), \phi_{\U}(S^\complement) )$ is a good 2-separation of the torso $M_v$. However, as $T'$ is a decomposition tree, the torso $M_v$ has no good 2-separations, a contradiction. \end{proof}

This concludes the proof of Lemma~\ref{lem:td-partial}.

%%%%%%%%%%%%%%%%%%%%%%%%%%%%%%%%%%%%%%%%%%%%%%%%%%%%%%%%%%%%

\section{The structure of the 2-separations of the dual}\label{sec:dual-main}
In this section, we prove Theorem~\ref{dual-main}. This essentially asserts that a tree-decomposition of a matroid is also a tree-decomposition of its dual. In particular, a matroid and its dual have the same unique irredundant tree-decomposition of uniform adhesion 2, and the corresponding torsos are duals of one another.  
This theorem is implied by Proposition~\ref{sep-dual} and Claim~\ref{localdual} stated below.
The former asserts that the $k$-separations of $M$ and its dual coincide.

\begin{proposition}\label{sep-dual}
Let $k\geq 1$ be an integer. A $k$-separation of a matroid $M$ is also a $k$-separation of its dual $M^*$. 
\end{proposition}

\noindent
This is well known for finite matroids~\cite{OxleyBook}. For infinite matroids this was established in~\cite[Lemma $18$]{BruhnWollanConInfMatroids}.
Here, we observe that Proposition~\ref{sep-dual} is a consequence of the following. 

\begin{claim}\label{dif-bases}
Fix $S \subseteq E(M)$ and let $B \in \B(M)$, $B^* = B^\complement \in \B(M^*)$. 
Extend $B \cap S$ and $B^* \cap S^\complement$ 
to bases $B_S \in \B(M|S)$ and $B^*_{S^\complement} \in \B(M^*|S^\complement)$, respectively. 
If $f = |B_S  - (B\cap S)| < \infty$, then $|B^*_{S^\complement} - (B^* \cap S^\complement)| = f$. 
\end{claim}

\noindent
To prove Claim~\ref{dif-bases} we use the following. 

\begin{lemma}\label{fin-def} \emph{\cite[Lemma $3.7$]{InfMatroidAxioms}}\\
If $B,B' \in \B(M)$ satisfy $|B \sm B'| < \infty$, then $|B \sm B'| = |B' \sm B|$. 
\end{lemma}

\begin{proofof}{Claim~\ref{dif-bases}}
Put $X = B_S \sm (B\cap S)$ and $Y = B^*_{S^\complement} \sm (B^* \cap S^\complement) = B^*_{S^\complement} \cap (B \cap S^\complement)$.
Noting that $M^*|S^\complement = M^* \sm S = (M/S)^*$, we have $B^*_{S^\complement} \in \B((M/S)^*)$ so $(B\cap S^\complement)\sm Y = S^\complement \sm B^*_{S^\complement} = E((M/S)^*) \sm B^*_{S^\complement} \in \B(M/S)$. By the definition of the contraction operation, it follows that $B' = B_S \cup (B \cap S^\complement) \sm Y \in \B(M)$. As $B' \sm B = X$ and $|X|= f < \infty$, then $f= |B' \sm B| = |B \sm B'| = |Y|$, by Lemma~\ref{fin-def}. 
\end{proofof}

Finally, we consider the torsos. For these we observe the following general property that holds for localizations (and not only for torsos). Let $M_{\U}$ be a localization of a connected matroid $M$ at $\U=\{X_i : i\in I\}$ where $(X_i, X_i^\complement)$ is a 2-separation of $M$ for all $i$.

\begin{claim}\label{localdual}
$(M^*)_{\U} = (M_{\U})^*$.
\end{claim}

\begin{proof}
Note that if $B$ is a base of $M$, then it follows from Claim~\ref{dif-bases} that for all $i\in I$, $B\cap X_i$ is a base of $M|X_i$ if and only if
$B^\complement \cap X_i$ is not a base of $M^*|X_i$.

Let $B_{\U}$ be a base of $M_{\U}$. There exists a base $B_M$ of $M$ such that $B_{\U} = (B_M \cap R(\U)) \cup \{e_i : B_M \cap X_i \in \B(M|X_i)\}$. Now $B_M^\complement$ is a base of $M^*$. Let $B_{\U}^*= (B_M^\complement \cap R(\U)) \cup \{e_i : B_M^\complement \cap X_i \in \B(M^*|X_i)\}$. Certainly, $B_{\U}^*$ is base of $(M^*)_{\U}$ and $B_{\U}^\complement$ is a base of $(M_{\U})^*$. 

We claim that $B_{\U}^* = B_{\U}^\complement$. It is straightforward to see that $B_U^* \cap R({\U}) = (B_{\U}\cap R({\U}))^\complement$. So Let $v=e_i$ for some $i$. Now $v$ is in $B_{\U}$ if and only if $B_M \cap X_i$ is a base of $M|X_i$. Similarly, $v$ is in $B_{\U}^*$ if and only if $B_M^\complement \cap X_i$ is a base of $M^*|X_i$. As noted above, $B_M \cap X_i$ is a base of $M|X_i$ if and only if $B_M^\complement \cap X_i$ is not a base of $M^*|X_i$. Thus, $v\in B_{\U}$ if and only if $v\not\in B_{\U}^*$ and the claim is proved. 
\end{proof}

%The fact that a tree-decomposition $(T,R)$ of a connected matroid $M$ is also a tree-decomposition $(T,R)$ of $M^*$ follows because a matroid and its dual have the same ground set. If such a tree-decomposition of $M$ has uniform adhesion~2, the same tree-decomposition for $M^*$ will also have uniform adhesion~2 by Proposition~\ref{sep-dual}. 
%
%Now, Proposition~\ref{sep-dual} and Claim~\ref{localdual} clearly imply Theorem~\ref{dual-main}; indeed,  
%$M$ and $M^*$ have the same unique irredundant tree-decomposition.

\bibliographystyle{amsplain}
\bibliography{collective}

\providecommand{\bysame}{\leavevmode\hbox to3em{\hrulefill}\thinspace}
\providecommand{\MR}{\relax\ifhmode\unskip\space\fi MR }
% \MRhref is called by the amsart/book/proc definition of \MR.
\providecommand{\MRhref}[2]{%
  \href{http://www.ams.org/mathscinet-getitem?mr=#1}{#2}
}
\providecommand{\href}[2]{#2}
\begin{thebibliography}{10}

\bibitem{AfzaliBowler12}
H.~Afzali and N.~Bowler, \emph{Thin sums matroids and duality}, Advances in
  Mathematics \textbf{271} (2015), 1--29, arXiv:1204.6294.

\bibitem{ALM15:InfiniteGammoids}
H.~Afzali, Hiu-Fai Law, and M.~M\"uller, \emph{Infinite gammoids}, Electron. J.
  Combin. \textbf{\#P1.53} (2015).

\bibitem{A-HCF11:matroidunion}
E.~{Aigner-Horev}, J.~Carmesin, and J-O. Fr\"ohlich, \emph{Infinite matroid
  union}, arXiv:1111.0602 (2011).

\bibitem{A-HCF11:matroidintersection}
\bysame, \emph{On the intersection of infinite matroids}, arXiv:1111.0606
  (2011).

\bibitem{BC12:excluded_matroid_minors}
N.~Bowler and J.~Carmesin, \emph{An excluded minors method for infinite
  matroids}, arXiv:1212.3939, 2012.

\bibitem{BC13:Determinacy}
\bysame, \emph{Infinite matroids and determinacy of games}, arXiv:1301.5980,
  2013.

\bibitem{BC13:Ubiquity}
\bysame, \emph{The ubiquity of {$\Psi$}-matroids}, arXiv:1304.6973, 2013.

\bibitem{BC14:Trees}
\bysame, \emph{Infinite trees of matroids}, arXiv:1409.6627, 2014.

\bibitem{BowlerCarmesinMatroidIntersection}
\bysame, \emph{Matroid intersection, base packing and base covering for
  infinite matroids}, Combinatorica (2014), 1--28, arXiv:1202.3409 (2012).

\bibitem{BC12:wildmatroids}
\bysame, \emph{Matroids with an infinite circuit-cocircuit intersection}, J.
  Combin. Theory Ser. B \textbf{107} (2014), 78--91, arXiv:1202.3409.

\bibitem{BC14:IntersectionTrees}
\bysame, \emph{On the intersection conjecture for infinite trees of matroids},
  arXiv:1404.6067, 2014.

\bibitem{BCC:graphic_matroids}
N.~Bowler, J.~Carmesin, and R.~Christian, \emph{Infinite graphic matroids,
  {Part~I}}, arXiv:1309.3735, 2013.

\bibitem{B13:TutteConnectivity}
H.~Bruhn, \emph{Matroid and {T}utte connectivity in infinite graphs},
  arXiv:1210.6380 (2012).

\bibitem{BruhnDiestelMatroidsGraphs}
H.~Bruhn and R.~Diestel, \emph{Infinite matroids in graphs}, Discrete Math.
  \textbf{311} (2011), 1461--1471, Special volume on {{\it Infinite Graphs:
  Introductions, Connections, Surveys\/} (R.~Diestel, G.~Hahn \& B.~Mohar,
  eds)}.

\bibitem{InfMatroidAxioms}
H.~Bruhn, R.~Diestel, M.~Kriesell, R.~Pendavingh, and P.~Wollan, \emph{Axioms
  for infinite matroids}, Advances in Mathematics \textbf{239} (2013), 18--46,
  arXiv:1003.3919.

\bibitem{BruhnWollanConInfMatroids}
H.~Bruhn and P.~Wollan, \emph{Finite connectivity in infinite matroids}, Europ.
  J. Comb. \textbf{33} (2012), 1900--1912.

\bibitem{C13:bureaucracy}
J.~Carmesin, \emph{Even an infinite bureaucracy eventually makes a decision},
  arXiv:1304.6973, 2013.

\bibitem{C14:CycleMatroids}
\bysame, \emph{Topological cycle matroids of infinite graphs}, arXiv:1412.0830,
  2014.

\bibitem{C14:gammoids}
\bysame, \emph{Topological infinite gammoids, and a new menger-type theorem for
  infinite graphs}, arXiv:1404.0151, 2014.

\bibitem{CE}
W.~H. Cunningham and J.~Edmonds, \emph{A combinatorial decomposition theory},
  Canad. J. Math. \textbf{32} (1980), no.~3, 734--765.

\bibitem{DiestelBook10noEE}
R.~Diestel, \emph{{Graph Theory}}, 4th ed., Springer, 2010.

\bibitem{OxleyBook}
J.G. Oxley, \emph{Matroid theory}, Oxford University Press, 1992.

\bibitem{R04}
Bruce Richter, \emph{Decomposing infinite 2-connected graphs into 3-connected
  components}, Electron. J. Combin. \textbf{11} (2004), no.~1.

\bibitem{Seymour}
P.~D. Seymour, \emph{Decomposition of regular matroids}, J. Combin. Theory Ser.
  B \textbf{28} (1980), no.~3, 305--359.

\end{thebibliography}

\end{document}